\documentclass[]{siamonline1116}

\ifpdf
\hypersetup{
  pdftitle={Lyapunov Function Partial Differential Equations for Chemical Reaction Networks: Some Special Cases},
  pdfauthor={Zhou Fang, and Chuanhou Gao}
}
\fi

\usepackage{amsmath}
\usepackage{mathrsfs,amssymb,amsmath,booktabs,array,xcolor,url,cite,color,soul,multirow,multicol}
\usepackage{mathbbold,bbm}
\usepackage{slashbox}
\usepackage{stfloats}
\usepackage{amsfonts}
\usepackage{extarrows}
\usepackage{textcomp}
\usepackage[normalem]{ulem}
\usepackage{cases}
\usepackage{bm,ulem}
\usepackage{arydshln}
\usepackage[all]{xy}
\usepackage{chemarrow}

\newtheorem{example}{\textit{Example}}
\newtheorem{remark}{\textit{Remark}}

\def\dd{\text{d}}

\title{Lyapunov Function Partial Differential Equations for Chemical Reaction Networks: {Some Special Cases}\thanks{This work is supported by National Natural Science Foundation of China under Grant No. 11671418, 11271326 and 61611130124, and the Research Fund for the Doctoral Program of Higher Education of China under Grant No. 20130101110040.}}

\author{Zhou Fang\footnotemark[2]
	\and Chuanhou Gao\footnotemark[2]}

\begin{document}

\maketitle

% REQUIRED
\begin{abstract}
In this paper we develop a {method} to generate the Lyapunov function for
stability analysis for chemical reaction networks.
{Based on the Chemical Master Equation, we derive the Lyapunov Function partial differential equations (PDEs), whose solution approximates the scaling non-equilibrium potential and serves as the candidate Lyapunov function for the given network.}
We further prove that for any chemical reaction
network the solution (if exists) of the PDEs is dissipative. Moreover, the proposed method of Lyapunov Function PDEs is qualified for analyzing the asymptotic stability of complex balanced networks, all networks with $1$-dimensional stoichiometric subspace and some special networks with more than $2$-dimensional stoichiometric subspace if some moderate conditions are added.
Several examples are presented to illustrate the
efficiency of the method.
\end{abstract}

% REQUIRED

% REQUIRED

\section{Introduction}
Chemical reactions networks (CRNs) arise abundantly in the fields
including chemistry, systems biology, process industry, and even
those seemingly irrelevant to chemistry such as mechanics and
ecology. The dynamics of a CRN often appears to be extremely
complex due to chemical interactions of cellular processes but
sometimes still exhibiting certain regular behaviors like period
solutions and stable-fixed-points. As a special subclass,
mass-action CRNs (CRNs assigned mass action kinetics, often named mass action systems) have the dynamics of the concentrations of the
various species captured by polynomial ordinary differential
equations (ODEs), and have received much attention since the pioneering
work \cite{Feinberg1972,Feinberg1987,Feinberg1988,Horn1972} emerged.
A major concern over this class of systems is to understand the relations between network structures and/or parameters and dynamical properties \cite{Angeli2007,Angeli2009,Szederkenyi2011}, especially in characterizing the stability (in the sense of Lyapunov) property \cite{Ali2014,Anderson2015,Horn1972,Schaft2013,Sontag2001,Rao2013}.
%{A major topic of this class of systems is to study their stable or unstable behaviors with the given network structures and parameters \mbox{\cite{Angeli2007,Angeli2009,Szederkenyi2011,Ali2014,Anderson2015,Horn1972,Schaft2013,Sontag2001,Rao2013}}.}
Following this line of study, we also focus on capturing stability of equilibria in mass action systems (MASs) in the current work. {Naturally, Lyapunov functions are desirable objects to prove stability of equilibria. In the field of CRNs, one important example is the pseudo-Helmholtz free energy function, proposed by Horn and Jackson\mbox{\cite{Horn1972}}. This Lyapunov function can be derived from the microscopic level using potential theory\mbox{\cite{Anderson2015}}. Here, we build further on this, and provide a general theory that bridges between the microscopic and the macroscopic level, thermodynamics and potential theory. Based on it, it is possible to derive or find a Lyapunov function for any MAS.}

{The early work on stability analysis mainly focused on exploring causal association from the network topology to the distribution of equilibria, and further to stability of equilibria.}
Thereinto, the weakly reversible structure, a requirement of complex balanced MAS, is the most active one.
{Horn et al.\mbox{\cite{Horn1972}} proved the well-known \textit{Deficiency Zero Theorem} that states a weakly reversible deficiency zero MAS to be complex balanced and to have only one equilibrium in each positive stoichiometric compatibility class. Moreover, each equilibrium in the complex balanced system is locally asymptotically stable, for which the pseudo-Helmholtz free energy function is proposed as the Lyapunov function. Feinberg\mbox{\cite{Feinberg1988}} extended this theorem to the well-known
\textit{Deficiency One Theorem} that suggests a weakly reversible
MAS to admit a sole equilibrium in each positive stoichiometric
compatibility class if some required conditions on network deficiency (not necessary to be zero) are satisfied. Based on these results, the global asymptotical stability of equilibria in a complex balanced MAS was further obtained\mbox{\cite{siegel2000global,Sontag2001}}} if the network is assumed to be persistent
\cite{Feinberg1987,Pantea2012,Craciun2013,Gopalkrishnan2014}, i.e., no stable boundary equilibrium if the initial point is in the interior of $\mathbb{R}^n_{\ge 0}$. Except for the weakly reversible structure, the reversible one, which acts as a special case of the former and is a requirement of detailed balanced MASs, is also the focus of attention. Feinberg \cite{Feinberg1989} derived necessary and sufficient
conditions, i.e., circuit conditions and spanning forest conditions,
to say a reversible MAS to be detailed balanced. Recently, van der Schaft et al.
\cite{Schaft2013} revisited this class of systems, and reported
a compact formulation to describe the dynamics utilizing the graph
theory (\textit{complex graph}). {The locally asymptotic stability of detailed balanced networks follows naturally from the fact that they are also complex balanced.
Still, the pseudo-Helmholtz free energy function serves as the Lyapunov function.}

{An important means for stability analysis of a MAS is to
construct a Lyapunov function according to the network structure.}
Although the {pseudo-Helmholtz free energy function} is capable for rendering
asymptotical stability of MASs equipped with the weakly reversible
or reversible structure, it fails to serve for those networks, {like
$3S_2\longrightarrow 3S_1\longrightarrow 2S_1+S_2$, where $S_1$ and
$S_2$ are the species. Clearly, this network is neither weakly reversible nor reversible. We name MASs with general structure (not necessarily weakly reversible or reversible) as balanced MASs if an equilibrium is admitted.} For balanced MASs, Angeli
and his coauthor \cite{Ali2014} proposed \textit{Piecewise Linear in Rates Lyapunov Functions} for stability analysis. The
existence of such functions (they defined the networks having this
attribute as $\mathscr{P}$ network set, which is actually a subset
of balanced MASs) can guarantee stability of equilibrium, and
further serves to establish asymptotical stability within the
corresponding positive stoichiometric compatibility class if the
Lyapunov function satisfies the LaSalle's condition. Another
possible solution to address the stability problem of a balanced MAS
comes from the concept of realization presented by Szederk\'enyi et.
al. \cite{Szederkenyi2011}. It is possible to find a complex balanced
or detailed balanced realization for the network in question, then
its asymptotical stability holds based on the dynamics equivalence
between the network and its realization.

{Different from} all of the above investigations stemming from macroscopic
deterministic analysis,
{some {literature contributes} to explaining the system properties from a microscopic stochastic viewpoint.
	Li and Yi \mbox{\cite{li2016systematic1,li2016systematic2}} connected the strength of attractions to the global attractor with the stationary distribution of a diffusion system, in which a white noise is added to the deterministic case.}
{In the meanwhile,}
Anderson et. al. \cite{Anderson2015}, { starting from a Markov chain model,} managed
to design a Lyapunov function from a microscopic stochastic concept
related to CRNs, termed as non-equilibrium potential that equals to
the minus logarithm of stationary distribution of state. They proved
that the scaling limit of non-equilibrium potential could act as a
Lyapunov function for some MASs. Moreover, such a limit value {coincides with} %is just
the well-known {pseudo-Helmholtz free energy function} in the case of complex
balanced MASs. This design thought is also valid for general
birth-death MASs and some examples of non-complex and non-detailed
balanced MASs.
These encouraging results motivate us to find a Lyapunov function for MASs rooted in their microscopic concepts.
In this paper, we take an approximation of the scaling non-equilibrium potential directly as a possible
Lyapunov function and carry out this idea on Chemical Master Equation.
A partial differential equation (PDE) is thus derived with
the solutions {serving} as %latent
{candidate} Lyapunov functions.
We have further proved
the equation solutions dissipative, and able to serve {as Lyapunov functions} for complex
balanced MASs, all networks with $1$-dimensional stoichiometric
subspace and some special networks with more than $2$-dimensional
stoichiometric subspace if some moderate conditions are added.

The {remainder} of this paper is organized as follows. Section $2$
revisits some basic concepts about CRNs and the macroscopic
deterministic dynamics of network derived from microscopic
stochastic model. This is followed by the development of the
Lyapunov Function PDEs in Section $3$. Section $4$ devotes to
analyzing the property of solutions of the Lyapunov Function PDEs,
and applications to complex balanced MASs. In Section $5$, we discuss
the validity of the Lyapunov Function PDEs for CRNs with
$1$-dimensional stoichiometric subspace, and further prove their efficacy in some special
examples of CRNs with more than $2$-dimensional stoichiometric
subspace in Section $6$. Finally, conclusions and a
conjecture to say the validity of the Lyapunov Function PDEs
to general balanced MASs are summarized in Section $7$.
~\\~\\
\noindent{\textbf{Mathematical Notation:}}\\
\rule[1ex]{\columnwidth}{0.8pt}
\begin{description}
	\item[\hspace{+0.8em}{$\mathbb{R}^n,\mathbb{R}^n_{\geq 0},\mathbb{R}^n_{\textgreater 0}$}]: $n$-dimensional real space, nonnegative and positive real space, respectively.
	\item[\hspace{+0.8em}{$x^{v_{\cdot i}}$}]: $x^{v_{\cdot i}}=\prod_{j=1}^{d}x_{j}^{v_{ji}}$, where $x,v_{\cdot i}\in\mathbb{R}^{d}$ and $0^{0}$ is defined to be $1$.
	\item[\hspace{+0.8em}{$\mathrm{Ln}(x)$}]: $\mathrm{Ln}(x)=\left(\ln{x_{1}},\ln{x_{2}},\cdots,\ln{{x}_{d}}\right)^{\top}$, where $x\in\mathbb{R}^{d}_{>{0}}$.
	\item[\hspace{+0.8em}{$\otimes$}]: Cartesian product.
	\item[\hspace{+0.8em}{$\mathscr{C}^{i}$}]: The function set whose elements are {$i$}-th continuous differentiable.
\item[\hspace{+0.8em}{$\mathbbold{0}_{n}$}]: {$n$-dimensional vector with every entry to be zero.}
\end{description}
\rule[1ex]{\columnwidth}{0.8pt}

\section{Preliminary on CRNs}In this section, we will sketch some
basic concepts about CRNs \cite{Feinberg1995} and revisit the
macroscopic dynamics of mass-action CRNs based on the microscopic
analysis \cite{Anderson2015}.
\subsection{Basic Concepts} Consider a network
{with} $n$ species, denoted by $\{S_1,\cdots,S_n\}$, and $r$
chemical reactions with the $i$th reaction $\mathcal{R}_i$ written
as $$\Sigma_{j=1}^{n} v_{ji}S_{j}\longrightarrow\Sigma_{j=1}^{n}
v'_{ji}S_{j},$$ where $v_{\cdot i}$, $v'_{\cdot i}\in
\mathbb{Z}^{n}_{\geq0}$ represent the complexes of reactant and
resultant, respectively, of this reaction. Note that we label each reaction as a
unidirectional reaction, here. If the $i$th reaction is reversible,
the reverse reaction is naturally covered by exchanging $v_{\cdot
	i}$ and $v'_{\cdot i}$ in the reaction.

Based on the above information, some basic concepts about CRNs may
be defined \cite{Feinberg1995}.

\begin{definition}[Chemical Reaction Network]\label{CRNdef}
	Denote the finite sets of
	species, complexes and reactions by $\mathcal{S}=\{S_{1},S_{2},\cdots,S_{n}\}$,
	$\mathcal{C}=\bigcup_{i=1}^{r} \{v_{\cdot i}, v'_{\cdot i}\}$ and
	$\mathcal{R}=\{v_{\cdot 1}\to v'_{\cdot 1}, \cdots, v_{\cdot r}\to v'_{\cdot r}\}$, respectively, and
	$\mathrm{Card}(\mathcal{C})=c$. If the following conditions hold
	\begin{description}
        \item[\hspace{1.5em}] {$\mathrm{(i)}$ There is no reaction $v_{\cdot i}\to v'_{\cdot i}\in \mathcal{R}$ ($i=1,\cdots,r$) such that $v_{\cdot i}=v'_{\cdot i}$;}
		\item[\hspace{1.5em}] $\mathrm{(ii)}$ The $j$th $(j=1,\cdots,n)$ entry of $v_{\cdot i}$		represents the stoichiometric coefficient of species $S_j\in\mathcal{S}$ in complex $v_{\cdot i}$,
	\end{description}
{then the triple {$(\mathcal{S},\mathcal{C},\mathcal{R})$}
		is called a chemical reaction network.}
\end{definition}

\begin{definition}[Stoichiometric Subspace]
	For a CRN {$(\mathcal{S},\mathcal{C},\mathcal{R})$}, the linear
	subspace $\mathscr{S}=\mathrm{span}\{v_{\cdot 1}-v'_{\cdot
		1},\cdots,v_{\cdot r}-v'_{\cdot r}\}$ is
	called the stoichiometric subspace of the network.
\end{definition}

\begin{definition}[Stoichiometric Compatibility Class]
	Let $\mathscr{S}$ be the stoichiometric subspace of a CRN
	{$(\mathcal{S},\mathcal{C},\mathcal{R})$} and $C \in
	\mathbb{R}^{n}_{\geq0}$ be a nonnegative $n$-dimensional vector,
	then $C+\mathscr{S}=\{C+\xi | \xi\in\mathscr{S}\}$ is a
	stoichiometric compatibility class of $C$ for the network;
	$(C+\mathscr{S})\bigcap\mathbb{R}^{n}_{\geq0}$ is a nonnegative
	stoichiometric compatibility class and
	$(C+\mathscr{S})\bigcap\mathbb{R}^{n}_{>0}$ is a positive
	stoichiometric compatibility class.
\end{definition}

When a CRN is assigned a mass action kinetics, the rate for reaction $v_{\cdot i}\to v'_{\cdot
	i}$ is evaluated by $k_ix^{v_{\cdot i}}$, where $k_i\in\mathbb{R}_{\textgreater 0}$ is the rate constant for this reaction, $x\in\mathbb{R}^n_{\geq 0}$ is the vector of concentrations $x_j$ of the chemical species $S_j,~j=1,\cdots,n$, and $$x^{v_{\cdot i}}:=\prod_{j=1}^{n}
x_{j}^{v_{ji}}.$$

\begin{definition}[Mass Action System]
	Denote the set of reaction rate constants by {$\mathcal{K}=(k_1,\cdots,k_r)$} with $k_i$ representing the rate constant for reaction $v_{\cdot i}\to v'_{\cdot i},~i=1,\cdots,r$. A CRN {$(\mathcal{S},\mathcal{C},\mathcal{R})$} taken together with the set of reaction rate constants $\mathcal{K}$ is called a mass action system, referred to as {$(\mathcal{S},\mathcal{C},\mathcal{R},\mathcal{K})$}.
\end{definition}

{The dynamics of a MAS $(\mathcal{S},\mathcal{C},\mathcal{R},\mathcal{K})$ that captures the changes of concentrations of every species over time $t$ is thus expressed as}
\begin{equation}\label{CRN-dynamics0}
\frac{\dd x}{\dd t}=\varGamma R(x),\quad\quad x\in\mathbb{R}^n_{\geq 0},
\end{equation}
{where $\varGamma\in\mathbb{Z}_{n\times r}$ is the stoichiometric matrix, defined by $\varGamma_{\cdot i}=v'_{\cdot i}-v_i$, and $R(x)$ is $r$-dimensional vector-valued function with $R_i(x)=k_ix^{v_{\cdot i}}$.}

\begin{definition}[Balanced MAS]
	For a MAS {$(\mathcal{S},\mathcal{C},\mathcal{R},\mathcal{K})$}, a vector of
	concentrations $x^*\in\mathbb{R}^n_{\textgreater 0}$ is called an
	equilibrium {if its dynamical equation \mbox{\cref{CRN-dynamics0}} satisfies $\varGamma R(x^*)=0$.} A MAS that admits an equilibrium is said to be a balanced MAS.
\end{definition}

\begin{definition}[Complex Balanced MAS]
	For a MAS {$(\mathcal{S},\mathcal{C},\mathcal{R},\mathcal{K})$}, a vector of
	concentrations $x^*\in\mathbb{R}^n_{\textgreater 0}$ is called a
	complex balanced equilibrium if at this state the combined rate of outgoing
	reactions from any complex is equal to the combined rate of incoming
	reactions to it, i.e.
	\begin{equation}\label{ComplexBalanecedCondition}
	  {\sum_{\{i|v_{\cdot i}=z\}} k_{i}(x^*)^{v_{\cdot i}}
	    =	\sum_{\{i|v'_{\cdot i}=z\}} k_{i}(x^*)^{v_{\cdot i}},
	    \qquad \forall~ z\in \mathcal{C}.}
	\end{equation}
	A MAS that admits a {complex balanced equilibrium} is said to
	be a complex balanced MAS.
\end{definition}

%Here, the reaction rate is a macroscopic concept that obeys thepower law in concentration variable for the mass action kinetics.
%We will revisit the deviation process of rate from its microscopical stochastic pattern to the macroscopical deterministic expression in the following.

{Eq. \mbox{\eqref{CRN-dynamics0}}
succeeds in modeling the CRN systems in the macroscopic level.
In the following subsection, we revisit its connection with the dynamic equation that models microscopic CRN systems.
}

\subsection{From Microscopic Stochastic Model to Macroscopic
	Deterministic Dynamics} In microscopic molecular level, a
{chemical reaction network system} is
commonly modeled by a continuous-time Markov chain, following which
every reaction takes place like a Poisson process
\cite{Anderson2011,Ethier2009}.
It allows for counting the
frequency that every reaction $\mathcal{R}_i~(i=1,\cdots,r)$ takes
place from initial time $0$ to time $t$ as
\begin{equation}{\label{ReactionFrequency}}
F_{i}(t)=\Omega_{i}\left(\int_{0}^{t}\lambda_{i}(N(\tau))\dd
\tau\right),
\end{equation}
where $N\in\mathbb{Z}^n_{\geq 0}$ is the vector of numbers of
molecules $N_j$ of every species $S_j,~j=1,\cdots,n$, indicating the state of the microscopic {system}, $\{\Omega_i(\cdot)\}_{i=1,\dots,r}$
are independent unit-rate Poisson processes that characterizes
reactions, and $\lambda_{i}(N)\in\mathbb{R}_{\geq 0}$ is a
intensity function reflecting the transition extent of
$\mathcal{R}_i$. The update for the state $N(t)$ is thus expressed,
according to {mass balance}, as
\begin{eqnarray}{\label{StochasticModel}}
N(t)&=&N(0)+\sum_{i=1}^{r}F_{i}(t)(v'_{\cdot i}-v_{\cdot i})\\
&=&N(0)+\sum_{i=1}^{r}\Omega_{i}\left(\int_{0}^{t}\lambda_{i}(N(\tau))\dd
\tau\right)(v'_{\cdot i}-v_{\cdot i}).\notag
\end{eqnarray}
This representation is often referred to as the stochastic model of
a {chemical reaction network system} in the
sense of microscopic level. Clearly, if $N(0)\in\mathbb{R}^n_{\geq
	0}$, then $N(t)\in(C+\mathscr{S})\bigcap\mathbb{R}^{n}_{\geq0}$.

One point should be noted that the model (\ref{StochasticModel})
only works up to the time {$\sup\{t|F_{i}(N(t))\textless
\infty, \forall i\}$, i.e., up to explosion of the process.}
{We thus restrict the subsequent discussion in case of non-explosive processes. In fact, this is not too strict for a Markov Chain. The explosive time can be almost surely infinite if some conditions are satisfied, such as that every transition intensity $\lambda_i(N(t))$ is bounded, and that the system is irreducible and finite time recurrent.} The mass action kinetics {indicates}
the intensity function to be modeled by
\begin{equation}{\label{IntensityFunction}}
\lambda_{i}(N)=\tilde{k}_{i}\prod_{j=1}^{n}\frac{N_{j}!}{(N_{j}-v_{ji})!}\mathbbold{1}_{\{N_{j}\geq v_{ji}\}},
\end{equation}
where $\tilde{k}_i\in\mathbb{R}_{\textgreater 0}$ is termed the microscopic rate constant and $\mathbbold{1}_{\{N_{j}\geq v_{ji}\}}$ is the characteristic
function, defined by
\begin{equation}
\mathbbold{1}_{\{N_{j}\geq v_{ji}\}}= \begin{cases}
1, & N_{j}\geq v_{ji},\\
0, &  N_{j}< v_{ji}.
\end{cases}
\end{equation}
{So long as the process is non-explosive}, the stochastic model of
(\ref{StochasticModel}) is equivalent to the corresponding
Kolmogorov's forward equation \cite{Anderson2011}, often called Chemical Master
Equation, that describes the probability distribution $P(N,t)$ of
$N(t)$ as
\begin{equation}{\label{CME}}
\frac{\dd P(N,t)}{\dd t}= \sum_{i=1}^{r} \lambda_{i}(N+v_{\cdot
	i}-v'_{\cdot i}) P(N+v_{\cdot i}-v'_{\cdot i},t) - P(N,t)
\sum_{i=1}^{r} \lambda_{i}(N).
\end{equation}

\begin{definition}[Stationary Distribution]
	A probability distribution $\pi(N)$ is a stationary distribution for
	the Markov chain on
	$(N(0)+\mathscr{S})\bigcap\mathbb{R}^{n}_{\geq0}$ if it satisfies
	\begin{equation}{\label{StationaryDistribution}}
	\sum_{i=1}^{r} \lambda_{i}(N+v_{\cdot i}-v'_{\cdot i}) \pi(N+v_{\cdot i}-v'_{\cdot i}) - \pi(N) \sum_{i=1}^{r} \lambda_{i}(N)=0,
	\end{equation}
	where $\pi(N+v_{\cdot i}-v'_{\cdot i})=0$ if $N+v_{\cdot
		i}-v'_{\cdot i}\notin
	(N(0)+\mathscr{S})\bigcap\mathbb{R}^{n}_{\geq0}$.
\end{definition}

The ergodic property of the continuous-time Markov chain states
\cite{Resnick2013} that if the chain on
$(N(0)+S)\bigcap\mathbb{R}^{d}_{\geq 0}$ is irreducible and
recurrent, then $\pi(N)$ {exists} and {is} unique.

We then revisit the macroscopic deterministic dynamics of the
underlying MAS derived from the microscopic stochastic model of
(\ref{StochasticModel}) by neglecting the random part under an
appropriate scaling level. %Consider
The differential form of
(\ref{StochasticModel}) can be divided into two parts{\mbox{\cite{oksendal2005applied}}}: the first one is
the deterministic part $\tilde{N}(t)$, i.e., the drift (expectation)
rate, satisfying
\begin{equation*}
\tilde{N}(t)\triangleq\lim_{\dd t\to 0^{+}}\frac{\text{E}\left[N(t+\dd
	t)-N(t)\big| N(t)\right]}{\dd t}=\sum_{i=1}^{r}
\lambda_{i}\big(N(t)\big)\big(v'_{\cdot i}-v_{i\cdot }\big),
\end{equation*}
while the second one is the random part $\hat{N}(t)$,
related to the following standard deviation rate
\begin{equation*}
\hat{N}(t)\triangleq\sqrt{\lim_{\dd t \to 0^{+}}\frac{{\text{Var}}\left[
	N(t+\dd t) - N(t) \big| N(t)\right]}{\dd
	t}}=\sqrt{\sum_{i=1}^{r}\lambda_{i}\big(N(t)\big)\big(v'_{\cdot i}-v_{\cdot
	i}\big)\big(v'_{\cdot i}-v_{\cdot i}\big)^\top},
\end{equation*}
{in which ``$\text{Var}$" is the variance operator.}
{Note that $\sum_{i=1}^{r}\lambda_{i}\big(N(t)\big)\big(v'_{\cdot i}-v_{\cdot
i}\big)\big(v'_{\cdot i}-v_{\cdot i}\big)^\top$ is positive semi-definite, so $\hat{N}(t)$ must exist.}

As the scale level increases, such as increasing from the molecular
level to molar level, the random part $\hat{N}(t)$, compared to the
deterministic part $\tilde{N}(t)$, contributes to the system smaller
and smaller, and can be ignored at last. Also, note the fact that $N_j(t)\gg v_{ji},~\forall i,j$, {then the stochastic model of \mbox{(\ref{StochasticModel})} can be well approximated by a deterministic one\mbox{\cite{anderson2015stochastic}}}
that describes the evolution of concentration vector $x(t)=\frac{N(t)}{A_v\cdot V}$, written as
\begin{equation}{\label{DeterministicModel}}
\frac{\dd x(t)}{\dd t} = \sum_{i=1}^{r} \tilde{k}_{i}
(A_{v}V)^{|v_{\cdot i}|-1} \left(\prod_{j=1}^{n}
x_{j}^{v_{ji}}\right) (v'_{\cdot i}-v_{\cdot i})
=\sum_{i=1}^{r} k_{i} x^{v_{\cdot i}} (v'_{\cdot i}-v_{\cdot i}),
\end{equation}
where $A_{v}$ is the Avogadro constant, $V$ the volume of the
system, $|v_{\cdot i}|$ indicates the sum of entries of vector
$v_{\cdot i}$, {$k_{i}=\tilde{k}_{i}(A_{v}V)^{|v_{\cdot i}|-1}$ is
the reaction rate coefficient for the $i$th reaction in the meaning
of macroscopic level} and $x^{v_{\cdot i}}=\prod_{j=1}^{n}
x_{j}^{v_{ji}}$ which indicates the mass-action kinetics. This expression is exactly the same as given in Eq. \mbox{(\ref{CRN-dynamics0})}. {More details about the derivation from the microscopic model to the macroscopic one may be referred to\mbox{\cite{anderson2015stochastic}}.}

One point needs to be noted that although it is a
	fact that the continuous-valued deterministic equation
	(\ref{DeterministicModel}) arises from the discrete probability
	model (\ref{CME}), the transformation between these two extremes are poorly
	understood. Some simulation analysis may be found in
	\cite{Higham2008}, and some connections between {deterministic models and stochastic counterpart} can be found in \cite{Kurtz1972,li2016systematic1,li2016systematic2}.
	
	\section{Lyapunov Function PDEs}
	This section contributes to deriving Lyapunov function PDEs for
	CRNs assigned mass action kinetics based on the relations between some microscopic
	concepts and macroscopic ones.
	\subsection{Lyapunov Function Derived from Stationary Distribution for Complex Balanced MASs}
	The CRNs theory \cite{Anderson2011,Ethier2009} reveals that there
	exists close relation between the microscopic stochastic dynamics
	and the macroscopic deterministic dynamics. Motivated by this fact,
	Anderson et al. \cite{Anderson2015} derived a Lyapunov
	function, a macroscopic concept, from the stationary distribution, a
	microscopic notion, for the stability analysis of complex balanced
	MASs.
	
	From the viewpoint of the macroscopic dynamics
	(\ref{DeterministicModel}), a MAS is complex balanced if
	$\exists~x^*\in\mathbb{R}^n_{\textgreater 0}$ such that for $\forall
	~z\in\mathcal{C}$ there is
	\begin{equation}
	\sum_{\{i| v_{\cdot i}=z\}} k_{i}\left(x^{*}\right)^{v_{\cdot i}}=
	\sum_{\{i| v'_{\cdot i}=z\}} k_{i}\left(x^{*}\right)^{v_{\cdot i}}.
	\end{equation}
	For this class of MASs, the {pseudo-Helmholtz free energy function}, defined by
	\begin{equation}{\label{Gibbs}}
	G(x)=\sum_{j=1}^n
	x_{j}\left(\ln(x_{j})-\ln(x^{*}_{j})-1\right)+x^{*}_{j},~~~x\in\mathbb{R}^n_{\textgreater
		0},
	\end{equation}
	is a frequently-used Lyapunov function \cite{Horn1972}.
	Despite a macroscopic concept, the {pseudo-Helmholtz free energy function} can be
	derived from the stationary distribution, a microscopic notion. As
	an example of a complex balanced MAS that admits an equilibrium of
	$x^*$ \cite{Anderson2010}, the stationary distribution $\pi(N)$ can be solved from
	(\ref{StationaryDistribution}) as
	\begin{equation*}
	\pi(N)=M\prod_{j=1}^{n}\frac{{x^{*}_{j}}^{N_{j}}}{N_{j}!},
	\end{equation*}
	where $M\in\mathbb{R}_{\textgreater 0}$ is a normalization factor. The
	non-equilibrium potential is thus expressed as
	$$-\mathrm{ln}(\pi(N))=-\mathrm{ln}M-\sum_{j=1}^{n}\mathrm{ln}\frac{{x^{*}_{j}}^{N_{j}}}{N_{j}!}. $$
	Further, Anderson et al. \cite{Anderson2015} proved that the
	scaling limit of non-equilibrium potential {coincides with} %is just right
	the {pseudo-Helmholtz free energy function}, i.e.,
	\begin{equation}
	\lim_{A_vV\to\infty} -\frac{1}{A_vV}\ln\big(\pi(A_{v}Vx)\big)=G(x).
	\end{equation}
	They also asserted that the scaling limit of non-equilibrium
	potential can suggest a Lyapunov function for the birth-death
	processes and some other special cases of non-complex balanced MASs
	\cite{Anderson2015}.
	
{Generally} speaking, the scaling limit of non-equilibrium potential
	provides a very effective way for some MASs to achieve the Lyapunov
	function with a definite physical meaning.
	However, it seems not easy to apply this {method} %way
	to more general MASs, {because solving the stationary Chemical Master Equation \mbox{\eqref{StationaryDistribution}} is usually a difficult task.} %The main reason is that solving (\ref{StationaryDistribution}) is usually a difficult task.
	To avoid this difficulty, %task,
	{we propose an alternative method, that is taking an approximation of the scaling non-equilibrium potential $-\frac{1}{A_vV}\ln\big(\pi(A_{v}Vx)\big)$ as a candidate Lyapunov function.}
	%an alternative way is to take an approximation of the scaling non-equilibrium potential $-\frac{1}{A_vV}\ln\big(\pi(A_{v}Vx)\big)$ directly as a candidate Lyapuonv function.
	Note that the former is naturally defined on a
	discrete set $\{x| A_{v}Vx\in \mathbb{Z}^{n}_{\geq 0}\}$ while the
	latter %requests to be
	is a continuous function defined on $\{x| x\in
	\mathbb{R}^{n}_{\geq 0}\}$.
	%Although this process seems not rigorous, it does not need, indeed, to know the explicit expression of the stationary distribution. {We only need to know that it exists.}
	{Obviously, the proposed method does not need to know the explicit expression of a stationary distribution, but only requires to know that a positive stationary distribution is existing.} We will follow this idea to derive Lyapunov function PDEs, and further solve Lyapunov functions for stability analysis of more general MASs below.
	
	\subsection{Derivation of Lyapunov Function PDEs}
	The approximation of the scaling non-equilibrium potential may be
	performed on the Chemical Master Equation (\ref{CME}) of a MAS
	{$(\mathcal{S},\mathcal{C},\mathcal{R},\mathcal{K})$}. Through dividing
	(\ref{CME}) by $-P(N,t)A_vV$, we can rewrite this equation as
	\begin{equation}\label{CME2}
	\frac{\dd}{\dd t} \left( - \frac{\ln\big(P(N,t)\big)}{A_{v}V}
	\right) =\sum_{i=1}^{r}\frac{\lambda_{i}(N)}{A_{v}V}-\sum_{i=1}^{r}
	\frac{\lambda_{i}(N+v_{\cdot i}-v'_{\cdot i})}{A_{v}V}
	\frac{P(N+v_{\cdot i}-v'_{\cdot i},t)}{P(N,t)},
	\end{equation}
	which is actually an ordinary differential equation about the
	scaling non-equilibrium potential. Assume that there exists a
	positive stationary distribution $\pi(N)$ for each {stoichiometric compatibility class} characterizing the MAS of interest, i.e.,
	the non-equilibrium potential exits. For simplicity
	of notations, denote the scaling non-equilibrium potential by
	$L(x)=-\frac{1}{A_vV}\ln\big(\pi(A_{v}Vx)\big)$, and then {by inserting}
	it into (\ref{CME2}) we get
	\begin{eqnarray}{\label{CME3}}
	&~&\frac{\dd L(x)}{\dd t}=0 \\
	~~~~~&=&\sum_{i=1}^{r}\frac{\lambda_{i}(A_{v}Vx)}{A_{v}V}-
	\frac{\lambda_{i}(A_{v}Vx+v_{\cdot i}-v'_{\cdot i})}{A_{v}V}
	\exp\left\{\frac{L(x)-L\left(x+\left(v_{\cdot i}-v'_{\cdot
			i}\right)/A_{v}V\right)}{1/A_{v}V}\right\}. \notag
	\end{eqnarray}
	
	Let {a continuous differentiable function} $f\in\mathscr{C}^1(\mathbb{R}^n_{\geq 0})$  approximate
	{the above function $L(x)$ ($x\in\{y|A_vVy\in\mathbb{Z}^n_{\geq 0}\}$).} %in the above equation,
	Then together with the fact $A_vV\gg
	(v_{ji}-v'_{ji}),~\forall~i,j$, it {indicates the exponential term in \mbox{\eqref{CME3}} to be approximated} as
	$$
	\exp\left\{\frac{f(x)-f\left(x+\left(v_{\cdot i}-v'_{\cdot
			i}\right)/A_{v}V\right)}{1/A_{v}V}\right\}\approx
	\exp\left\{(v'_{\cdot i}-v_{\cdot i})^{\top}\nabla f(x)\right\}.
	$$
    The remaining two terms $\frac{\lambda_{i}(A_{v}Vx)}{A_{v}V}$
	and $\frac{\lambda_{i}(A_{v}Vx+v_{\cdot i}-v'_{\cdot i})}{A_{v}V}$ can be thought as the same in the macroscopic coordinated and be approximated by
	$k_{i}x^{v_{\cdot i}}$. As a result, the Chemical Master Equation of
	(\ref{CME3}) {becomes} a first-order partial differential
	equation
	\begin{equation}{\label{LyaPDEs}}
	\sum_{i=1}^{r} k_{i}x^{v_{\cdot i}}-\sum_{i=1}^{r} k_{i}x^{v_{\cdot
			i}} \exp\left\{(v'_{\cdot i}-v_{\cdot i})^{\top}\nabla
	f(x)\right\}=0,~~~{x\in\mathbb{R}^n_{\textgreater 0}}.
	\end{equation}
{We can alternatively express this PDE according to the complexes set}
\begin{equation}{\label{LyaPDEs1}}
\sum_{\{i|v_{\cdot i}\in\mathcal{C}\}} k_{i}x^{v_{\cdot i}}-\sum_{\{i|v'_{\cdot i}\in\mathcal{C}\}} k_{i}x^{v_{\cdot i}} \exp\left\{(v'_{\cdot i}-v_{\cdot i})^{\top}\nabla
	f(x)\right\} =0,~~~{x\in\mathbb{R}^n_{\textgreater 0}}.
\end{equation}

Note that the above PDE \mbox{(\ref{LyaPDEs})} or \mbox{(\ref{LyaPDEs1})} is derived from the Chemical Master Equation by setting the solution as an approximation of the scaling non-equilibrium potential. Its existence seems dependent on the existence of the non-equilibrium potential, i.e., on that of a stationary distribution. {Although it is quite difficult to know whether a stationary distribution is existing in \mbox{(\ref{StationaryDistribution})}, it will not limit the applicability of the developed theory. We find that the Lyapunov Function PDE \mbox{(\ref{LyaPDEs})} or \mbox{(\ref{LyaPDEs1})} can be also achieved for some systems without the non-equilibrium potential. For example, the following CRN with absorption}
\begin{eqnarray*}
		&& S_1+2S_2 \to 3S_2,\\
		&& 2S_2 \to S_1+S_2,
	\end{eqnarray*}
{there will be eventually one of the species $S_2$ in which case none of the reactions can fire, so the potential does not exist. However, we can write out its Lyapunov Function PDE according to \mbox{(\ref{LyaPDEs})} or \mbox{(\ref{LyaPDEs1})}.} {A reasonable explanation may be that the solution of \mbox{(\ref{StationaryDistribution})} for this network would be to consider the QSD (quasi-stationary distribution) or a modification where the CRN cannot jump to the state with $S_2=1$. The latter has been done by Anderson et. al.\mbox{\cite{Anderson2015}} for birth-death processes with absorption. We thus stipulate that for those networks without a stationary distribution, the solution of \mbox{(\ref{StationaryDistribution})} would be to consider the QSD or a modification of the rates if the PDE is derived using potential theory.} In fact, the PDE may be also generated directly from the macroscopic dynamics of the CRN under study. In this {sense}, there always exists a corresponding PDE \eqref{LyaPDEs} for a CRN {no matter whether} the non-equilibrium potential {is existing or not}.

	To solve a PDE, it usually needs to know the related boundary
	conditions. For the above one, we still derive its boundary conditions
    based on the approximation to the Chemical Master
	Equation. Since it is very hard to directly analyze the boundary
	conditions for (\ref{LyaPDEs}), we manage to get an insight into
	them through the following example of a special MAS.
	
	\begin{example}
		Consider a MAS including a first-order reversible reaction	$S_{1}\leftrightharpoons S_{2}$. {The species set is $\mathcal{S}=\{S_1,S_2\}$, the complex set having the same form $\mathcal{C}=\{S_1,S_2\}$, and the reaction set is $\mathcal{R}=\{S_1\to S_2,S_2\to S_1\}$. Using the notations given in \mbox{\cref{CRNdef}}, the last two sets might be written as $\mathcal{C}=\bigcup_{i=1}^{2} \{v_{\cdot i}, v'_{\cdot i}\}$ and $\mathcal{R}=\{v_{\cdot 1}\to v'_{\cdot 1},v_{\cdot 2}\to v'_{\cdot 2}\}$,} where $v_{\cdot 1}=v'_{\cdot 2}=(1,0)^{\top}$, and $v_{\cdot 2}=v'_{\cdot 1}=(0,1)^{\top}$. {The domain of the stochastic model for this network is a nonnegative discrete set $\{x=(x_1,x_2)^\top|~A_{v}Vx_1,A_{v}Vx_{2}\in\mathbb{Z}_{>0}\}$, where $x$ is the vector of molar concentration. It thus defines two subsets of boundary points, denoted by $\mathcal{M}_1=\{x|~x_1=0,A_{v}Vx_{2}\in\mathbb{Z}_{>0}\}$ and $\mathcal{M}_2=\{x|~A_{v}Vx_{1}\in\mathbb{Z}_{>0},x_2=0\}$, respectively.}

{As an example of a boundary point $\bar{x}\in\mathcal{M}_1$, the last state of $\bar{x}$, just before the latest reaction, might be $A_{v}V\bar{x}+v_{\cdot 1}-v'_{\cdot 1}=(1,A_{v}Vx_{2}-1)^\top$ or $A_{v}V\bar{x}+$ $v_{\cdot 2}-v'_{\cdot 2}=(-1,A_{v}Vx_{2}+1)^\top$. By substituting these three states into Eq. \mbox{(\ref{IntensityFunction})}, we can calculate intensity functions as follows.}
\begin{equation*}
		\lambda_1(A_vV\bar{x})=\lambda_2(A_vV\bar{x}+v_{\cdot 2}-v'_{\cdot 2})=0,~\lambda_1(A_{v}V\bar{x}+v_{\cdot 1}-v'_{\cdot 1})\neq 0,~\lambda_{2}(A_{v}V\bar{x})\neq 0.
		\end{equation*}
	Further, by inserting these intensity functions into Eq. (\mbox{\ref{CME3}}), we may get a boundary condition for the Chemical Master Equation
		of this MAS as
		\begin{equation*}
		\frac{\lambda_{2}(A_{v}V\bar{x})}{A_{v}V}-\frac{\lambda_{1}(A_{v}V\bar{x}+v_{\cdot
				1}-v'_{\cdot 1})}{A_{v}V}
		\exp\left\{\frac{L(\bar{x})-L\left(\bar{x}+(v_{\cdot 1}-v'_{\cdot
				1})/A_{v}V\right)}{1/A_{v}V}\right\}=0, \quad \bar{x}\in \mathcal{M}_1.
		\end{equation*}
		{Finally, by} the approximation scheme from (\ref{CME3}) to (\ref{LyaPDEs}),
		the above boundary condition can be approximated {by}	\begin{equation}\label{BoundaryConditoinExample1}
		\lim_{x\to \bar{x},~x\in\mathbb{R}^2_{\textgreater 0}}
		k_{2}x^{v_{\cdot 2}}-k_{1}x^{v_{\cdot 1}}\exp\left\{(v'_{\cdot
			1}-v_{\cdot 1})^{\top}\nabla f(x)\right\}=0, \qquad \bar{x}\in\{0\}\times\mathbb{R}_{>0},
		\end{equation}
		{which serves as a boundary condition for the PDE of the corresponding system.}
		{Similarly, the boundary condition at $\bar{x}\in\mathcal{M}_2$ is presented as }
		\begin{equation}\label{BoundaryConditoinExample1_2}
		\lim_{x\to \bar{x},~x\in\mathbb{R}^2_{\textgreater 0}}
		k_{1}x^{v_{\cdot 1}}-k_{2}x^{v_{\cdot 2}}\exp\left\{(v'_{\cdot
			2}-v_{\cdot 2})^{\top}\nabla f(x)\right\}=0, \qquad \bar{x}\in\mathbb{R}_{>0}\times\{0\}.
		\end{equation}
		\end{example}
	
	The above example provides a clear insight into how to express the
	boundary conditions for the PDE (\ref{LyaPDEs}) or (\ref{LyaPDEs1}), i.e., identifying
	non-zero intensity functions with the given boundary points set. By Eq. \mbox{(\ref{IntensityFunction})}, whether or not the reaction's intensity function $\lambda_i(\cdot)$ is zero depends closely on its complexes.
	{At any boundary point $\bar{x}$, Eq. \mbox{\eqref{IntensityFunction}} tells us that the intensities of reactions with the same reactant complex are simultaneously positive or zero. We call the set of complexes which generate positive intensities at boundary point $\bar{x}$ as a boundary complex set of $\bar{x}$, and denoted it by $\mathcal{C}_{\bar{x}}$ in the context. The boundary complex set may vary from point to point. Also, from Eq. \mbox{\eqref{IntensityFunction}}, we can easily find that the intensity $\lambda_{i}(A_vV\bar{x}+v_{\cdot i}-v'_{\cdot i})$ is positive, only if the resultant complex of the corresponding reaction $v'_{\cdot i}$ lies in the boundary complex set of $\bar{x}$. In \textit{Example 1}, at the boundary point $(0,x_{2})^{\top}$, only the reaction with reactant complex $(0,1)^{\top}$ has positive intensity function. The boundary complex set of $(0,x_{2})^{\top}$ is thus to be $\{(0,1)^{\top}\}$. Based on the same analysis, the boundary complex set of $(x_1,0)^{\top}$ is $\{(1,0)^{\top}\}$. Also, at the boundary point $(0,x_{2})^{\top}$, the intensity function $\lambda_2(A_vV\bar{x}+v_{\cdot 2}-v'_{\cdot2})>0$ because the resultant complex in the second reaction lies in the corresponding boundary complex set.}

    {With these understandings, we can rewrite Eq. \mbox{\eqref{CME3}} at any boundary point $\bar{x}$ as}
    \begin{equation*}
	\sum_{\{i| v_{\cdot i}\in\mathcal{C}_{\bar{x}}\}}
	\frac{\lambda_{i}(A_{v}V\bar{x})}{A_{v}V}-
	\sum_{\{i| v'_{\cdot i}\in\mathcal{C}_{\bar{x}}\}}
	\frac{\lambda_{i}(A_{v}V\bar{x}+v_{\cdot i}-v'_{\cdot i})}{A_{v}V}
	\exp\left\{\frac{L(\bar{x})-L\left(\bar{x}+\left(v_{\cdot i}-v'_{\cdot i}\right)/A_{v}V\right)}{1/A_{v}V}\right\}=0. \notag
    \end{equation*}
    {By applying the same approximation scheme used above, we arrive at the boundary condition of the developed PDE \mbox{\eqref{LyaPDEs}}}
    \begin{equation}{\label{BoundaryCondition}}
	\lim_{
		\begin{tiny}
		\begin{array}{c}
		x \to \bar{x}\\
		x\in (\bar{x}+\mathscr{S})\cap\mathbb{R}^{n}_{\textgreater 0}
		\end{array}
		\end{tiny}
	}\sum_{\{i| v_{\cdot i}\in\mathcal{C}_{\bar{x}}\}} k_{i}x^{v_{\cdot
			i}}
	-\sum_{\{i| v'_{\cdot i}\in \mathcal{C}_{\bar{x}}\}} k_{i}x^{v_{\cdot i}}\exp\{(v'_{\cdot i}-v_{\cdot i})^{\top}\nabla f(x)\}=0,
	\end{equation}
    where $\bar{x}$ is any boundary point lies in the union of all stoichiometric compatibility classes, $\{x\in(y+\mathscr{S})\cap\mathbb{R}^{n}_{\geq0}~|~ y\in \mathbb{R}^{n}_{>0} ~\text{and}~x\notin\mathbb{R}^{n}_{>0}\}.$
	{Here, the limit notation is introduced to make the terms well defined in the case where $\nabla f(\cdot)$ does not converge at the boundary point.}
	{In \textit{Example 1}, the boundary condition \mbox{\eqref{BoundaryConditoinExample1}} can be written as}
	\begin{equation*}{\label{BoundaryConditionExample2}}
	\lim_{
		\begin{tiny}
		\begin{array}{c}
		x \to \bar{x},~x\in\mathbb{R}^2_{\textgreater 0}
		\end{array}
		\end{tiny}
	}\sum_{\{i| v_{\cdot i}\in\{(0,1)^\top\}\}} k_{i}x^{v_{\cdot
			i}}
	-\sum_{\{i| v'_{\cdot i}\in \{(0,1)^\top\}\}} k_{i}x^{v_{\cdot i}}\exp\{(v'_{\cdot i}-v_{\cdot i})^{\top}\nabla f(x)\}=0.
	\end{equation*}
	{which falls into the expression \mbox{\eqref{BoundaryCondition}} and illustrates the correctness of our derivation.}

Clearly, identifying the boundary complex set $\mathcal{C}_{\bar{x}}$ plays a key role on
	formulating the boundary conditions. Generally
	speaking, it is not easy to identify $\mathcal{C}_{\bar{x}}$,
	especially when the underlying CRN is {complicated}. A possible expression
	for it may be obtained from revisiting (\ref{IntensityFunction})
	where positive intensity function requests $N_j\textgreater
	v_{ji},~\forall~i,j$. We thus {can express a particular boundary complex set} as
	\begin{equation}{\label{naive boundary complex set}}
	\bar{\mathcal{C}}_{\bar{x}}=\left\{z\in\mathcal{C} ~|~
	\exists~\epsilon\textgreater 0~\text{such that}~\forall
	j=1,\cdots,n,~\bar{x}_{j}\geq \epsilon z_{j} \right\}
	\end{equation}
	which is referred to as naive boundary complex set in the context.

	The PDE (\ref{LyaPDEs}) and its boundary condition
	(\ref{BoundaryCondition}) will serve for generating the Lyapunov
	function for macroscopic deterministic mass-action CRNs. They are
	referred to as \textit{Lyapunov Function PDEs} throughout the paper.
	
	\section{Solutions of Lyapunov Function PDEs}
	This section focuses on analyzing the property and utility %special usage
	of
	solutions of Lyapunov Function PDEs if {they exist}.
	\subsection{Conditions for Solutions to Become Lyapunov Function}
	We firstly analyze the dissipativeness of solutions of the Lyapunov Function PDEs
	(\ref{LyaPDEs}) plus (\ref{BoundaryCondition}), a necessary property for solutions becoming Lyapunov functions, under the assumption that the solutions exist.
	
	\begin{theorem}{\label{Dissipativeness} For a MAS
			{$(\mathcal{S},\mathcal{C},\mathcal{R},\mathcal{K})$}} described by
		(\ref{DeterministicModel}), assume that there exists a solution
		{$f\in\mathscr{C}^{1}$} for its Lyapunov Function PDEs (\ref{LyaPDEs}) plus
		(\ref{BoundaryCondition}). Then, the solution $f(x)$ satisfies
		\begin{equation}{\label{eq dissipation of f}}
		\dot{f}\left(x\right)=\frac{\dd f\left(x\right)}{\dd t}\leq 0, ~~~~~~~~\forall~
		x\in\mathbb{R}^{n}_{\textgreater 0},
		\end{equation}
		where the equality holds if and only if $\nabla f(x) \perp
		\mathscr{S}$.
	\end{theorem}
	
	\begin{proof}
		Reorganize the PDE $(\ref{LyaPDEs})$ to be
		\begin{equation*}
		\sum_{i=1}^{r}k_{i}x^{v_{\cdot i}}\left(1-\exp\left\{(v'_{\cdot
			i}-v_{\cdot i})^{\top}\nabla f(x)\right\}\right)=0
		\end{equation*}
		and further perform the Taylor expansion of $\exp\left\{(v'_{\cdot
			i}-v_{\cdot i})^{\top}\nabla f(x)\right\}$ with respect to zero,
		then we have
		\begin{equation*}
		\sum_{i=1}^{r}k_{i}x^{v_{\cdot i}}(v'_{\cdot i}-v_{\cdot
			i})^{\top}\nabla f(x) +\sum_{i=1}^{r}k_{i}x^{v_{\cdot
				i}}\frac{\text{e}^{\eta_{i}}}{2}\big[(v'_{\cdot i}-v_{\cdot
			i})^{\top}\nabla f(x)\big]^{2} =0,
		\end{equation*}
		where $\eta_{i}\in\mathbb{R}$ lies between $0$ and $(v'_{\cdot
			i}-v_{\cdot i})^{\top}\nabla f(x)$. Since for $\forall~
		x\in\mathbb{R}^{n}_{>0}$ there is
		$$\dot{f}(x)=\dot{x}^\top\nabla f(x)=\sum_{i=1}^{r}k_{i}x^{v_{\cdot i}}(v'_{\cdot i}-v_{\cdot
			i})^{\top}\nabla f(x),$$ we get
		$$\dot{f}(x)=-\sum_{i=1}^{r}k_{i}x^{v_{\cdot
				i}}\frac{\text{e}^{\eta_{i}}}{2}\big[(v'_{\cdot i}-v_{\cdot
			i})^{\top}\nabla f(x)\big]^{2}\leq 0,$$
		where the equality holds if and only if for $\forall~i,~(v'_{\cdot i}-v_{\cdot
			i})^{\top}\nabla f(x)=0$, i.e., $\nabla f(x) \perp \mathscr{S}$.
	\end{proof}
	
	\begin{remark}
		The dissipativeness of $f(x)$ means that it has one {of the necessary}
		properties to become a Lyapunov function. In addition, this property
		implies that $-f(x)$ will always increase as time goes by, which
		further indicates that there may be a close relation between $-f(x)$
		and the entropy function, an important concept in thermodynamics. A
		possible point of future research may be to define or derive the
		entropy expression based on the Lyapunov Function PDEs instead of
		the Gibbs' Equation.
	\end{remark}
	
	We further derive the conditions that the {non-dissipative} point of $f(x)$ is the equilibrium point of the MAS.
	
	\begin{theorem}{\label{thm strict dissipation}}
		For a MAS {$(\mathcal{S},\mathcal{C},\mathcal{R},\mathcal{K})$} described by
		(\ref{DeterministicModel}), assume that  its Lyapunov Function PDEs (\ref{LyaPDEs}) and
		(\ref{BoundaryCondition}) admit a solution
		{$f\in\mathscr{C}^{2}$}, and moreover, there exists a region $D\subset\mathbb{R}^{n}_{>0}$ such that $\forall~x\in D$ and $\forall~ \mu \in \mathscr{S}$ {we have}
		\begin{equation}{\label{eq. strictly convex condition}}
		\begin{array}{c}
		\mu^{\top} \nabla^{2} f(x) \mu \geq 0~~ \text{with equality hold if and only if $\mu=\mathbbold{0}_{n}$}.
		\end{array}
		\end{equation}
		Then, for all $x\in D$, $\dot{f}(x)=0$ if and only if $x$ is an equilibrium of the MAS.
	\end{theorem}
	\begin{proof}
		The necessity is obvious. For the sufficiency, \textit{Theorem \ref{Dissipativeness}} suggests that {for any $x\in D$,} {$\dot{f}(x)=0$} if and only if $\nabla f(x)\perp \mathscr{S}$. By taking the derivative of \eqref{LyaPDEs} with respect to $x$ on both sides, and further inserting the condition $\nabla f(x)\perp \mathscr{S}$, we have
		\begin{equation*}
		\nabla^{2} f(x)
		\left[\sum_{i=1}^{r} k_{i}x^{v_{\cdot i}} (v'_{\cdot i}-v_{\cdot i})\right] = \mathbbold{0}_{n},
		\end{equation*}
		i.e.,
		\begin{equation*}
		\left[\sum_{i=1}^{r} k_{i}x^{v_{\cdot i}} (v'_{\cdot i}-v_{\cdot i})\right]^{\top}
		\nabla^{2} f(x)
		\left[\sum_{i=1}^{r} k_{i}x^{v_{\cdot i}} (v'_{\cdot i}-v_{\cdot i})\right] = 0.
		\end{equation*}
		Note that the term $\sum_{1}^{r} k_{i}x^{v_{\cdot i}} (v'_{\cdot i}-v_{\cdot i})$ lies in $\mathscr{S}$, so we get $\sum_{1}^{r} k_{i}x^{v_{\cdot i}} (v'_{\cdot i}-v_{\cdot i})=\mathbbold{0}_{n}$ from the condition \eqref{eq. strictly convex condition}, which means that $x\in D$ should be an equilibrium of the MAS. This completes the proof.
	\end{proof}
	
	\begin{remark}
		Theorem \ref{thm strict dissipation} reveals that {for a balanced MAS {$(\mathcal{S},\mathcal{C},\mathcal{R},\mathcal{K})$} the solution of its Lyapunov Function PDEs (if exists) is strictly dissipative and, therefore, a good candidate {for a} Lyapunov function, provided that the solution is twice differentiable and convex in $D\bigcap\big(x^{*}+\mathscr{S}\big)\bigcap\mathbb{R}^{n}_{\geq 0}$.}
	\end{remark}
	
	Finally, we give conditions which indicates the solution to be indeed a Lyapunov function.
	
	\begin{theorem}{\label{thm lya function}}
		For a MAS {$(\mathcal{S},\mathcal{C},\mathcal{R},\mathcal{K})$} governed by
		(\ref{DeterministicModel}), let $x^{*}\in\mathbb{R}^n_{\textgreater 0}$ be one of its equilibrium points. Assume that the Lyapunov Function PDEs (\ref{LyaPDEs}) and
		(\ref{BoundaryCondition}) of the MAS admit a solution
		{$f\in\mathscr{C}^{2}$}, and moreover, there exists a region {$D=\mathcal{N}(x^*)=\delta(x^{*})\bigcap\big(x^{*}+\mathscr{S}\big)\bigcap\mathbb{R}^{n}_{\textgreater 0}$}, where $\delta(x^{*})$ is a neighborhood of $x^*$, such that $\forall~x\in\mathcal{N}(x^*)$ the solution $f(x)$ satisfies (\ref{eq. strictly convex condition}). Then $f(x)$ can act as a Lyapunov function rendering $x^{*}$ to be locally asymptotically stable with respect to all initial conditions in $\mathcal{N}(x^*)\bigcap\{x|f(x)\textless \inf_{\{y\in\partial_{\mathcal{N}(x^*)}\}}f(y)\}$.
	\end{theorem}
	\begin{proof}
		Since $f(x)$ satisfies (\ref{eq. strictly convex condition}) in $\mathcal{N}(x^*)$,
		$f(x)$ is strictly convex in this region.
		{The strict convexity together with the fact, $\nabla f(x^{*}) \perp \mathscr{S}$ (by \textit{Theorem \mbox{\ref{Dissipativeness}}}), implies the function to be lower bounded by $f(x^{*})$.}
		{Also, the strict convexity suggests that no other state except $x^{*}$ can make $\nabla f(x) \perp \mathscr{S}$ and, therefore, that $x^{*}$ is the sole equilibrium in this region (by \textit{Theorem \mbox{\ref{Dissipativeness}}}).}
		{Thus, by \textit{theorem \mbox{\ref{thm strict dissipation}}}, this fact states $\dot{f}(x)\leq 0$ with equality hold if and only if $x=x^*$.}
		
		For any initial point $x(0)\in\mathcal{N}(x^*)\bigcap\{x|f(x)\textless \inf_{\{y\in\partial_{\mathcal{N}(x^*)}\}}f(y)\}$, since $\dot{f}(x)\leq 0$, the state trajectory of the mass action system starting from $x(0)$ will be bounded in the region $\mathcal{N}(x^*)\bigcap\{x|f(x)\textless \inf_{\{y\in\partial_{\mathcal{N}(x^*)}\}}f(y)\}$. Therefore, if $f(x)$ is selected as the Lyapunov function, then $x^{*}$ is locally asymptotically stable with respect to all initial conditions in $\mathcal{N}(x^*)\bigcap\{x|f(x)\textless \inf_{\{y\in\partial_{\mathcal{N}(x^*)}\}}f(y)\}$.
	\end{proof}
	
	It is clear that the Lyapunov Function PDEs \mbox{\eqref{LyaPDEs}} and
	\mbox{\eqref{BoundaryCondition}} {have potentials} to generate a solution {serving as} the Lyapunov function for MASs with some moderate conditions {satisfied}.
	We try our hands at a class of special MASs, i.e., complex balanced MASs, to test the method of the PDEs in the following.
	\subsection{Test on Complex Balanced MASs}
	We will demonstrate that the Lyapunov Function PDEs {work} for
	complex balanced MASs.
	As mentioned in Section {$3.1$}, a complex balanced MAS admits an
	equilibrium $x^*$ satisfying the relation \eqref{ComplexBalanecedCondition}.
	Moreover, the equilibrium was proved locally asymptotically stable
	through taking the {pseudo-Helmholtz free energy function} as the Lyapunov function
	\cite{Horn1972,Rao2013}.
	To show the power of Lyapunov function PDEs,
	we verify that the {pseudo-Helmholtz free energy function} is one of their solutions.
	
	%An immediate idea is to verify that the {pseudo-Helmholtz free energy function} is a solution of the Lyapunov Function PDEs, %which inturn indicates that the Lyapunov Function PDEs {work} for complexbalanced MASs.
	
	\begin{theorem}{\label{thm solution of complex balanced crn}}
		For a MAS {$(\mathcal{S},\mathcal{C},\mathcal{R},\mathcal{K})$} that admits a
		complex balanced equilibrium $x^*\in\mathbb{R}^n_{\textgreater 0}$,
		the {pseudo-Helmholtz free energy function} of (\ref{Gibbs}) is a solution of the
		corresponding Lyapuonv Function PDEs (\ref{LyaPDEs}) (or equivalently \eqref{LyaPDEs1}) and
		(\ref{BoundaryCondition}) whatever the boundary complex set is.
	\end{theorem}
	
	\begin{proof}
		For $\forall~ x\in
		(x(0)+\mathscr{S})\cap\mathbb{R}^{n}_{\textgreater 0}$, the gradient
		of the {pseudo-Helmholtz free energy function} is
		\begin{equation*}
		\nabla G(x)= \text{Ln}\left(\frac{~x~}{~x^{*}}\right)
		=\left(\ln\left(\frac{x_{1}}{x^{*}_{1}}\right),\ln\left(\frac{x_{2}}{x^{*}_{2}}\right),\cdots,\ln\left(\frac{x_{n}}{x^{*}_{n}}\right)\right)^{\top}.
		\end{equation*}
		Plugging it into the left hand side (L.H.S) of (\ref{LyaPDEs1}) yields
		\begin{eqnarray*}
			\text{L.H.S of Eq. (\ref{LyaPDEs1})}&=&\sum_{\{i|v_{\cdot i}\in\mathcal{C}\}}
			k_{i}x^{v_{\cdot i}}-\sum_{\{i|v'_{\cdot i}\in\mathcal{C}\}} k_{i}x^{v_{\cdot i}}
			\exp\left\{(v'_{\cdot i}-v_{\cdot i})^{\top}\text{Ln}\left(\frac{~x~}{~x^{*}}\right)\right\}\\
			&=& \sum_{z\in\mathcal{C}}\left(\sum_{\{i|v_{\cdot i}=z\}} k_{i}x^{v_{\cdot i}}-\sum_{\{i|v'_{\cdot i}=z\}} k_{i}x^{v_{\cdot i}} \exp\left\{(v'_{\cdot i}-v_{\cdot i})^{\top}\text{Ln}\left(\frac{~x~}{~x^{*}}\right)\right\} \right)\\
			&=& \sum_{z\in\mathcal{C}}\left( \sum_{\{i|v_{\cdot i}=z\}}
			k_{i}x^{v_{\cdot
					i}}-\left(\frac{~x~}{~x^{*}}\right)^{z}\sum_{\{i|v'_{\cdot i}=z\}}
			k_{i}(x^*)^{v_{\cdot i}} \right)\\
			&=& \sum_{z\in\mathcal{C}} \left(\frac{~x~}{~x^{*}}\right)^{z} \cdot  \left(\sum_{\{i|v_{\cdot i}=z\}}
			k_{i} (x^{*})^{v_{\cdot i}}- \sum_{\{i|v'_{\cdot i}=z\}}
			k_{i}(x^*)^{v_{\cdot i}} \right)
			=0,
		\end{eqnarray*}
		where the last equality follows immediately from {the complex balanced condition \mbox{\eqref{ComplexBalanecedCondition}}}. Hence, $G(x)$ satisfies the PDE of \eqref{LyaPDEs1} and (\ref{LyaPDEs}).
		
		We further verify that $G(x)$ satisfies the boundary condition of
		(\ref{BoundaryCondition}). Let $\mathcal{C}_{\bar{x}}$ be a boundary
		complex set induced by any boundary point $\bar{x}$, then the left
		hand side of (\ref{BoundaryCondition}) is
		\begin{eqnarray*}
			& & \lim_{
				\begin{tiny}
					\begin{array}{c}
						x \to \bar{x}\\
						x\in (x(0)+\mathscr{S})\cap\mathbb{R}^{n}_{\textgreater 0}
					\end{array}
				\end{tiny}
			}\sum_{\{i| v_{\cdot i}\in\mathcal{C}_{\bar{x}}\}} k_{i}x^{v_{\cdot
					i}} -\sum_{\{i| v'_{\cdot i}\in \mathcal{C}_{\bar{x}}\}}
			k_{i}x^{v_{\cdot i}}\exp\left\{(v'_{\cdot i}-v_{\cdot i})^{\top}\text{Ln}\left(\frac{~x~}{~x^{*}}\right)\right\}\\
			&=& \lim_{
				\begin{tiny}
					\begin{array}{c}
						x \to \bar{x}\\
						x\in (x(0)+\mathscr{S})\cap\mathbb{R}^{n}_{\textgreater 0}
					\end{array}
				\end{tiny}
			} \sum_{z\in\mathcal{C}_{\bar{x}}}\left(\sum_{\{i|v_{\cdot i}=z\}} k_{i}x^{v_{\cdot i}}-\sum_{\{i|v'_{\cdot i}=z\}} k_{i}x^{v_{\cdot i}} \exp\left\{(v'_{\cdot i}-v_{\cdot i})^{\top}\text{Ln}\left(\frac{~x~}{~x^{*}}\right)\right\} \right)\\
			&=&\lim_{
				\begin{tiny}
					\begin{array}{c}
						x \to \bar{x}\\
						x\in (x(0)+\mathscr{S})\cap\mathbb{R}^{n}_{\textgreater 0}
					\end{array}
			\end{tiny}} \sum_{z\in\mathcal{C}_{\bar{x}}}
		\left(\frac{~x~}{~x^{*}}\right)^{z} \cdot  \left(\sum_{\{i|v_{\cdot i}=z\}}
		k_{i} (x^{*})^{v_{\cdot i}}- \sum_{\{i|v'_{\cdot i}=z\}}
		k_{i}(x^*)^{v_{\cdot i}} \right)
		=0.
		\end{eqnarray*}
		Note that {the above equations hold} independent of the choice of
		$\mathcal{C}_{\bar{x}}$, which completes the proof.
	\end{proof}
	
	{It is well-known that} the {pseudo-Helmholtz free energy function} is a Lyapunov
	function for a complex balanced MAS {and succeeds in analyzing the system's asymptotic stability}
	\cite{Horn1972}. This stability result can be also reached through the method of the Lyapunov Function PDEs.
	
	\begin{theorem}{\label{thm stability of complex balanced CRN}}
		For a MAS {$(\mathcal{S},\mathcal{C},\mathcal{R},\mathcal{K})$} possessing a
		complex balanced equilibrium $x^*\in\mathbb{R}^n_{\textgreater 0}$,
		the Lyapunov function PDEs (\ref{LyaPDEs}) plus
		(\ref{BoundaryCondition}) have a solution \eqref{Gibbs} {that can serve as a} Lyapunov
		function to suggest this system to be locally asymptotically stable at $x^*$ with respect to all initial conditions in $\big(x^{*}+\mathscr{S}\big)\bigcap\mathbb{R}^{n}_{\textgreater 0}$ near $x^*$.
	\end{theorem}
	
	\begin{proof}
		As proved in \textit{Theorem \ref{thm solution of complex balanced crn}}, the {pseudo-Helmholtz free energy function} $G(x)$
		defined by (\ref{Gibbs}) is a solution of the Lyapunov function PDEs
		(\ref{LyaPDEs}) plus (\ref{BoundaryCondition}).
		Obviously, $G(x)$ is twice differentiable, and its Hessian matrix is calculated as
		\begin{equation*}
		\nabla^{2} G(x)
		=\left(
		\begin{array}{ccc}
		\frac{1}{x_{1}}&&\\
		&\ddots&\\
		&& \frac{1}{x_{n}}
		\end{array}
		\right).
		\end{equation*}
		Clearly, $\forall~x\in\big(x^{*}+\mathscr{S}\big)\bigcap\mathbb{R}^{n}_{\textgreater 0}$, $\nabla^{2} G(x)$ is positive definite. This means that (\ref{eq. strictly convex condition}) is true. Further based on \textit{Theorem \ref{thm lya function}}, the result is straightforward.
	\end{proof}
	
	The above two theorems reveal that the Lyapunov Function PDEs method can {produce Lyapunov functions \mbox{\eqref{Gibbs}} for complex balanced MASs and serve for the stability analysis of these systems very well}.
	In this case, the Lyapunov Function PDE \mbox{(\ref{LyaPDEs1})} becomes
\begin{equation}{\label{complexLyaPDEs}}
\sum_{z\in\mathcal{C}}\text{e}^{z^\top\nabla G(x)}\left(\sum_{\{i|v_{\cdot i}=z\}} k_{i}x^{v_{\cdot i}} \text{e}^{-v_{\cdot i}^{\top}\nabla
	G(x)}-\sum_{\{i|v'_{\cdot i}=z\}} k_{i}x^{v_{\cdot i}} \text{e}^{-v_{\cdot i}^{\top}\nabla
	G(x)}\right) =0,~~~x\in\mathbb{R}^n_{\textgreater 0}.
\end{equation}
{Further by combining \mbox{(\ref{ComplexBalanecedCondition})}, we get $k_{i}x^{v_{\cdot i}}=k_{i}(x^*)^{v_{\cdot i}}\exp\left\{v_{\cdot i}^{\top}\nabla G(x)\right\}$. This relational expression can be also found in Gorboan's work\mbox{\cite{Gorban2014}}, which connects the reaction rate at any concentration with that at the equilibrium concentration through the entropy-like function $G(x)$. When $\nabla G(x)=\mathbbold{0}_n$, every complex will reach reaction balance. At this point, $\nabla G(x)$ plays a role on driving the reaction to occur towards equilibrium for every complex.}

	\section{Lyapunov Function PDEs for CRNS with $\text{dim}\mathscr{S}=1$}
	The Lyapunov function PDEs are {studied} %extended to serve
	for CRNs with one
	dimensional stoichiometric subspace in this section.
	
	\begin{proposition} For a MAS
		{$(\mathcal{S},\mathcal{C},\mathcal{R},\mathcal{K})$} with
		$\text{dim}\mathscr{S}=1$, the Lyapunov Function PDEs are
		\begin{equation}\label{LyaS1}
		(u-1)
		\left[\sum_{\{i|m_{i}>0\}}\left(k_{i}x^{v_{\cdot
				i}}\right)\left(\sum_{j=0}^{m_{i}-1}u^{j}\right)
		+\sum_{\{i|m_{i}<0\}}\left(k_{i}x^{v_{\cdot
				i}}\right)\left(-\sum_{j=m_{i}}^{-1}u^{j}\right)\right]=0
		\end{equation}
		plus the boundary condition
		\begin{equation}{\label{BoundaryConditionS1}}
		\lim_{
			\begin{tiny}
			\begin{array}{c}
			x \to \bar{x}\\
			x\in (\bar{x}+\mathscr{S})\cap\mathbb{R}^{n}_{\textgreater 0}
			\end{array}
			\end{tiny}
		}\sum_{\{i| v_{\cdot i}\in\mathcal{C}_{\bar{x}}\}} k_{i}x^{v_{\cdot
				i}} -\sum_{\{i| v'_{\cdot i}\in \mathcal{C}_{\bar{x}}\}}
		k_{i}x^{v_{\cdot i}}u^{m_i}=0,
		\end{equation}
		where
		{$u=\exp\{w^{\top}\nabla f\}$, {$w\in\mathbb{R}^n\backslash\{\mathbbold{0}_{n}\}$ represents a set of bases of $\mathscr{S}$} and $m_i\in\mathbb{Z}\backslash\{0\},~i=1,\cdots,r$, satisfy
			\begin{equation}\label{LinearExpression}
			v'_{\cdot i}-v_{\cdot
				i}=m_{i}w.
			\end{equation}
		}
	\end{proposition}
	\begin{proof}
		When $\text{dim}\mathscr{S}=1$, any element among $\{v'_{\cdot
			1}-v_{\cdot 1},\cdots,v'_{\cdot r}-v_{\cdot r}\}$ can be used to
		express linearly the remaining $r-1$ ones. Therefore, there must
		exist a $w\in\mathbb{R}^n\backslash\{\mathbbold{0}_{n}\}$ acting as
		a set of bases of $\mathscr{S}$ such that
		{
			\begin{equation*}
			v'_{\cdot i}-v_{\cdot
				i}=m_{i}w,~\forall~i=1,\cdots,r,~m_i\in\mathbb{Z}\backslash\{0\}.
			\end{equation*}
		}
		In this case, the PDE of (\ref{LyaPDEs}) {becomes}
		$$\sum_{i=1}^{r} k_{i}x^{v_{\cdot i}}\left(1-e^{m_i(w^{\top}\nabla f(x))}\right)=0,$$
		i.e.,
		$$\sum_{\{i|m_{i}>0\}} k_{i}x^{v_{\cdot i}}\left(1-e^{m_i(w^{\top}\nabla f(x))}\right)+\sum_{\{i|m_{i}<0\}}k_{i}x^{v_{\cdot i}}\left(1-e^{m_i(w^{\top}\nabla f(x))}\right)=0.$$
		By setting $u=e^{w^{\top}\nabla f(x)}$, we get the Lyapunov function
		PDEs (\ref{LyaS1}) plus (\ref{BoundaryConditionS1}) for CRNs with
		$\text{dim}\mathscr{S}=1$.
	\end{proof}
	
	\begin{corollary}
		For any constant $\mathfrak{c}$, $f(x)=\mathfrak{c}$ is a solution
		of the PDE (\ref{LyaS1}).
	\end{corollary}
	\begin{proof}
		The result is immediate since $u=1$ is a solution of the PDE.
	\end{proof}
	
	\begin{remark}
		If $f(x)=\mathfrak{c}$ satisfies the boundary condition of
		(\ref{BoundaryConditionS1}), then
		$$\lim_{
			\begin{tiny}
			\begin{array}{c}
			x \to \bar{x}\\
			x\in (x(0)+\mathscr{S})\cap\mathbb{R}^{n}_{\textgreater 0}
			\end{array}
			\end{tiny}
		}\sum_{\{i| v_{\cdot i}\in\mathcal{C}_{\bar{x}}\}} k_{i}x^{v_{\cdot
				i}} -\sum_{\{i| v'_{\cdot i}\in \mathcal{C}_{\bar{x}}\}}
		k_{i}x^{v_{\cdot i}}=0$$
		%This requests the underlying CRN to be more special but not only limited to admitting one dimensional stoichiometric subspace.
		{This condition is very restrictive that can be hardly reached even in the case of one dimensional stoichiometric subspace.}
		Therefore, the constant solution $f(x)=\mathfrak{c}$ is
		usually not qualified to follow the boundary condition
		(\ref{BoundaryConditionS1}) for a general CRN with
		$\text{dim}\mathscr{S}=1$.
	\end{remark}
	
	The above reason motivates us to {consider} {the solution}
	that makes the second term of the \text{L.H.S} of (\ref{LyaS1}) vanish,
	i.e.,
	$$\sum_{\{i|m_{i}>0\}}\left(k_{i}x^{v_{\cdot
			i}}\right)\left(\sum_{j=0}^{m_{i}-1}u^{j}\right)
	+\sum_{\{i|m_{i}<0\}}\left(k_{i}x^{v_{\cdot
			i}}\right)\left(-\sum_{j=m_{i}}^{-1}u^{j}\right)=0.$$
	
	\begin{proposition}{\label{uExistence}}
		For a MAS {$(\mathcal{S},\mathcal{C},\mathcal{R},\mathcal{K})$} with
		$\text{dim}\mathscr{S}=1$, let a scalar function $g(x,u)$ defined on
		$\mathbb{R}^{n}_{\textgreater 0} \times \mathbb{R}_{\textgreater 0}$
		be
		\begin{equation}{\label{gDefinition}}
		g(x,u)=\sum_{\{i|m_{i}>0\}}\left(k_{i}x^{v_{\cdot
				i}}\right)\left(\sum_{j=0}^{m_{i}-1}u^{j}\right)+\sum_{\{i|m_{i}<0\}}\left(k_{i}x^{v_{\cdot
				i}}\right)\left(-\sum_{j=m_{i}}^{-1}u^{j}\right).
		\end{equation}
		If the MAS admits a positive steady state
		$x^{*}\in\mathbb{R}^{n}_{\textgreater 0}$, then there exists a
		unique {$\tilde{u}\in\mathscr{C}^2$} such that
		$g(x,\tilde{u}(x))=0$.
	\end{proposition}
	\begin{proof}
		Since the MAS admits a positive steady state $x^{*}$, its dynamics
		satisfies
		\begin{eqnarray*}
			\dot{x}|_{x=x^{*}}&=&\sum_{i=1}^{r} k_{i} (x^{*})^{v_{\cdot i}}
			(v'_{\cdot i}-v_{\cdot i})\\
			&=&\sum_{\{i|m_{i}>0\}} k_{i}
			(x^{*})^{v_{\cdot i}} (m_{i}w)-\sum_{\{i|m_{i}<0\}} k_{i}
			(x^{*})^{v_{\cdot i}} (|m_{i}|w)=0,
		\end{eqnarray*}
		which indicates that neither $\{i|m_{i}>0\}$ nor $\{i|m_{i}<0\}$ is
		an empty set. Combing this fact and the definition
		(\ref{gDefinition}) of $g(x,u)$ yields that $g(x,u)$ is continuous
		in $\mathbb{R}^{n}_{\textgreater 0} \times \mathbb{R}_{\textgreater
			0}$, and moreover for $\forall~x\in\mathbb{R}^{n}_{>0}$, $g(x,u)$ is
		continuous differentiable about $u$ with
		\begin{equation*}
		\frac{\partial}{\partial u}g(x,u)=
		\sum_{\{i|m_{i}>0\}}\left(k_{i}x^{v_{\cdot
				i}}\right)\left(\sum_{j=1}^{m_{i}-1}j
		u^{j-1}\right)+\sum_{\{i|m_{i}<0\}}\left(k_{i}x^{v_{\cdot
				i}}\right)\left(\sum_{j=m_{i}}^{-1}(-j)u^{j-1}\right)>0.
		\end{equation*}
		Hence, $g(x,u)$ is monotone increasing over $u$. Also, note the
		facts that $$\lim_{u\to0} g(x,u)=-\infty~
		\text{and}~\lim_{u\to+\infty} g(x,u)=+\infty,$$ then based on the
		intermediate value theorem there exists a unique
		$\tilde{u}(x)\in\mathbb{R}_{\textgreater 0}$ such that
		$g(x,\tilde{u}(x))=0$.
		In addition, $g(x,u)$ is also continuous
		differentiable about $x$ and $g_u(x,u)|_{u=\tilde{u}}\neq 0$, so
		we have {$\tilde{u}\in\mathscr{C}^1(\mathbb{R}^n_{\textgreater
			0};\mathbb{R}_{\textgreater 0})$} according to the implicit function
		theorem and $\nabla\tilde{u}(x)=-{g_{x}(x,\tilde{u}(x))}/{g_{u}(x,\tilde{u}(x))}$.
		Moreover, since the functions $g_{x}$ and $g_{u}$ are also continuous differentiable with respect to both parameters, the function $\nabla\tilde{u}(x)$ is also continuous differentiable and therefore {$\tilde{u}\in\mathscr{C}^{2}(\mathbb{R}^n_{\textgreater
			0};\mathbb{R}_{\textgreater 0})$}, which completes the proof.
	\end{proof}
	
	Based on the function $\tilde{u}(x)$, we could find a solution for
	the Lyapunov function PDEs (\ref{LyaS1}) plus
	(\ref{BoundaryConditionS1}) derived from a MAS with
	$\text{dim}\mathscr{S}=1$ and a positive equilibrium. For this
	purpose, we begin with the following two lemmas.
	
	\begin{lemma}{\label{ZxPwNw}}
		For a MAS {$(\mathcal{S},\mathcal{C},\mathcal{R},\mathcal{K})$} with
		$\text{dim}\mathscr{S}=1$, let
		$w\in\mathbb{R}^n\backslash\{\mathbbold{0}_{n}\}$ be a set of bases
		of $\mathscr{S}$, and $\bar{x}\in\mathbb{R}^n_{\geq 0}$ represent
		any boundary point of any positive stoichiometric compatibility
		class induced by $\mathscr{S}$. Denote the index sets of positive
		and negative entries of $w$ by $P_{w}$ and $N_w$, respectively, and
		the index set of zero entries of $\bar{x}$ by $Z_{\bar{x}}$, then
		for $\forall~\bar{x}$,
		$$Z_{\bar{x}}\subseteq P_{w}~~ \mathrm{or}~~ Z_{\bar{x}}\subseteq
		N_{w}.$$
	\end{lemma}
	
	\begin{proof}
		Since $\bar{x}$ is a boundary point of a positive stoichiometric
		compatibility class induced by $\mathscr{S}$, so there exists a
		nonzero constant $\alpha\in\mathbb{R}$ such that
		\begin{equation*}
		\bar{x}+\alpha w \in \mathbb{R}^{n}_{\textgreater 0}.
		\end{equation*}
		If $\alpha\textgreater 0$, then for $\forall~ i\in Z_{\bar{x}}$
		($Z_{\bar{x}}$ is obviously non-empty) we have
		\begin{equation*}
		\bar{x}_{i}+\alpha w_{i}\textgreater 0 \Rightarrow w_{i}\textgreater
		0 \Rightarrow i\in P_{w}.
		\end{equation*}
		Therefore, $Z_{\bar{x}}\subseteq P_{w}$. Similarly, if
		$\alpha\textless 0$ then we {get} $Z_{\bar{x}}\subseteq N_{w}$.
	\end{proof}
	
	\begin{lemma}{\label{CoordinateTransformation}}
		For a MAS {$(\mathcal{S},\mathcal{C},\mathcal{R},\mathcal{K})$} with
		$\text{dim}\mathscr{S}=1$, a function $J(y)$ from positive
		stoichiometric compatibility class
		$(x+\mathscr{S})\bigcap\mathbb{R}^{n}_{\textgreater 0}$ to
		$\mathbb{R}$ is defined as follows
		\begin{equation}{\label{J(y)}} J(y)=\left\{
		\begin{array}{ll}
		\prod_{i\in P_{w}} y_{i} - \prod_{i\in N_{w}} y_{i}, &~ P_{w},N_{w}\neq \emptyset,\\
		\prod_{i\in P_{w}} y_{i} - 1,  &~ P_{w} \neq \emptyset, N_{w} = \emptyset,\\
		\prod_{i\in N_{w}} y_{i} - 1,  & ~N_{w} \neq \emptyset, P_{w} = \emptyset,\\
		\end{array}
		\right.
		\end{equation}
		where $x\in\mathbb{R}_{\textgreater 0}^n$ represents any state of
		the MAS. Then, this function $J(y)$ admits a unique zero point in
		every $(x+\mathscr{S})\bigcap\mathbb{R}^{n}_{\textgreater 0}$, and
		moreover, the unique zero point is a twice continuous differential
		function with respect to $x$, denoted by
		{$y^{\dag}\in\mathscr{C}^2(\mathbb{R}^n_{\textgreater
			0};\mathbb{R}_{\textgreater 0}^n)$}. In addition, there also exists
		another twice continuous differential function
		{$\gamma\in\mathscr{C}^2(\mathbb{R}^n_{\textgreater
			0};\mathbb{R})$}, which together with $y^{\dag}(x)$ satisfies
		$$x=y^{\dag}(x)+\gamma(x)w~~\text{and}~~\gamma(x+\delta w)=\gamma(x)+\delta,~\forall~\delta\in\mathbb{R},$$
		where $w$ is a set of bases of $\mathscr{S}$.
	\end{lemma}
	
	\begin{proof}
		{Clearly, for
			$\forall~y\in(x+\mathscr{S})\bigcap\mathbb{R}^{n}_{>0}$ there exists
			a boundary point $\bar{x}$ and $\beta\in\mathbb{R},~\beta\neq 0$
			such that $y=\bar{x}+\beta w$. We conduct the proof according to
			three different cases below:}
		
		1) $P_{w}\neq\emptyset$ and $N_{w}\neq\emptyset$. In every
		$(x+\mathscr{S})\bigcap\mathbb{R}^{n}_{\textgreater 0}$, there may
		exist two distinct boundary points, denoted by $\bar{x}$ and
		$\bar{x}'$, and moreover, they could reach each other through
		$\bar{x}'=\bar{x}+\beta_Mw$, where $\beta_M\in\mathbb{R}$ but
		$\beta_M\neq 0$. According to \textit{Lemma \ref{ZxPwNw}}, for $\bar{x}$
		either $Z_{\bar{x}} \subseteq P_{w}$ or $Z_{\bar{x}} \subseteq
		N_{w}$ is true. If $Z_{\bar{x}} \subseteq P_{w}$, then we have
		$\beta_M\textgreater 0$ and $\beta\in(0,\beta_M)$. Further, we have
		$Z_{\bar{x}'} \subset N_{w}$. Following these results, we get
		\begin{equation*}
		\lim_{\beta\to 0}J(\bar{x}+\beta w)\textless 0~~~\text{and}~~~
		\lim_{\beta\to \beta_{M}}J(\bar{x}+\beta w)\textgreater 0.
		\end{equation*}
		By {the intermediate value theorem}, there exist a point
		$y^\dag\in(x+\mathscr{S})\bigcap\mathbb{R}^{n}_{\textgreater 0}$
		rendering $J(y^\dag)=0$. Also, we note that $\frac{d}{d\beta}
		J(\bar{x}+\beta w)=\left(\frac{\partial J}{\partial
			y}\right)^\top\frac{\partial y}{\partial \beta}\textgreater 0$, the
		zero point $y^\dag$ is unique. Similarly, if $Z_{\bar{x}} \subseteq
		N_{w}$, the result is true too.
		
		2) $P_{w} \neq \emptyset$ and $N_{w} = \emptyset$. In this case
		$Z_{\bar{x}} \subseteq P_{w}$ and $\alpha \in (0,+\infty)$. Thus, we
		have
		\begin{equation*}
		\lim_{\beta\to 0}J(\bar{x}+\beta w)=-1,~~~ \lim_{\beta\to
			+\infty}J(\bar{x}+\beta w)=+\infty~~~\text{and}~~~ \frac{d}{d \beta}
		J(\bar{x}+\beta w)\textgreater 0.
		\end{equation*}
		According to {the intermediate value theorem} and strict monotonicity,
		$J(y)$ admits a unique zero point in
		$(x+\mathscr{S})\bigcap\mathbb{R}^{n}_{\textgreater 0}$.
		
		3) $N_{w} \neq \emptyset$ and $P_{w} = \emptyset$. Based on the
		similar reason as in case 2), we can get the result immediately.
		
		We continue to prove
		{$y^{\dag}\in\mathscr{C}^2(\mathbb{R}^n_{\textgreater
			0};\mathbb{R}_{\textgreater 0}^n)$} and
		{$\gamma\in\mathscr{C}^2(\mathbb{R}^n_{\textgreater
			0};\mathbb{R})$}. Let the function
		$$\tilde{J}(x,\beta)=J(x-\beta w),~x-\beta w\in\mathbb{R}^{n}_{\textgreater 0}.$$
		Clearly, for $\forall~x\in\mathbb{R}^n_{\textgreater 0}$ there
		exists a sole $\beta=\gamma(x)$ such that $J(x-\gamma(x) w)=0$,
		i.e., $x-\gamma(x)w=y^\dag(x)$ and $\tilde{J}(x,\gamma(x))=0$. Note
		that $\tilde{J}(x,\beta)$ is continuous differentiable from the
		definition of $J(y)$ and $\frac{\partial}{\partial
			\beta}\tilde{J}(x,\beta)\Big|_{\beta=\gamma(x)}\neq0$, then by
		{the implicit function theorem} we have
		{$\gamma\in\mathscr{C}^1(\mathbb{R}^n_{\textgreater
			0};\mathbb{R})$} and $\nabla\gamma(x)=-{\tilde{J}_{x}(x,\gamma(x))}/{\tilde{J}_{\beta}(x,\gamma(x))}$.
		In addition, since the function $\tilde{J}_{x}$ and $\tilde{J}_{\beta}$ are also continuous differentiable with respect to both parameters, we can conclude that $\nabla\gamma(x)$ is continuous differentiable and therefore {$\gamma\in\mathscr{C}^2(\mathbb{R}^n_{\textgreater0};\mathbb{R})$}. Further, we get
		{$y^{\dag}\in\mathscr{C}^2(\mathbb{R}^n_{\textgreater
			0};\mathbb{R}_{\textgreater 0}^n)$} from $y^\dag(x)=x-\gamma(x)w$,
		i.e., $x=y^\dag(x)+\gamma(x)w$.
		
		Finally, we focus on proving $\gamma(x+\delta
		w)=\gamma(x)+\delta,~\forall~\delta\in\mathbb{R}$. Since
		$(x+\mathscr{S})\bigcap\mathbb{R}^{n}_{\textgreater 0}=(x+\delta
		w+\mathscr{S})\bigcap\mathbb{R}^{n}_{\textgreater 0}$, we have
		$y^\dag(x)=y^\dag(x+\delta w)$, i.e., $x+\delta w-\gamma(x+\delta
		w)w=x-\gamma(x)w$. Further, since $w$ is a set of bases of
		$\mathscr{S}$, we get $\gamma(x+\delta
		w)=\gamma(x)+\delta,~\forall~\delta$.
	\end{proof}
	
	By means of $\tilde{u}(x),~y^\dag(x)$ and $\gamma(x)$, a solution
	for the Lyapunov function PDEs (\ref{LyaS1}) plus
	(\ref{BoundaryConditionS1}) is reachable.
	
	\begin{theorem}{\label{f(x)ExistenceS1}} For a MAS {$(\mathcal{S},\mathcal{C},\mathcal{R},\mathcal{K})$} with
		$\text{dim}\mathscr{S}=1$ and a positive steady state,
		{$\bar{\mathcal{C}}_{\bar{x}}$} in the form of (\ref{naive boundary complex
			set}) is selected as the boundary complex set where $\bar{x}$ is any
		boundary point of any positive stoichiometric compatibility class
		induced by $\mathscr{S}$.
		Assume that
		$\bar{\mathcal{C}}_{\bar{x}}=\emptyset$ or
		$\bar{\mathcal{C}}_{\bar{x}}$ includes at least a reactant complex and a resultant complex, then the function defined by
		\begin{equation}{\label{SolutionOfS1}}
		f(x)= \int_{0}^{\gamma(x)} \ln{\tilde{u}(y^{\dag}(x)+ \tau w)} \dd
		\tau
		\end{equation}
		is a solution of the Lyapunov function PDEs (\ref{LyaS1}) plus
		(\ref{BoundaryConditionS1}), where $y^{\dag}(x)$, $\gamma(x)$ and
		$w$ share the same meanings with those in \textit{Lemma \ref{CoordinateTransformation}}.
	\end{theorem}
	
	\begin{proof}
		(1) The first part serves for proving that $f(x)$ in the form of
		(\ref{SolutionOfS1}) satisfies (\ref{LyaS1}). Since {$y^{\dag},\gamma,\tilde{u}\in\mathscr{C}^{2}$}, $f(x)$ is obviously a twice
		continuous differentiable function defined on $\mathbb{R}^{n}_{>0}$.
		Thus, we have
		\begin{eqnarray*}
			w^\top\nabla f(x)&=&\lim_{\delta\to 0} \frac{f(x+\delta
				w)-f(x)}{\delta}\\
			&=&\lim_{\delta\to 0}
			\frac{1}{\delta}\int_{\gamma(x)}^{\gamma(x+\delta w)}
			\ln{\tilde{u}(y^{\dag}(x)+ \tau w)} \dd \tau\\
			&=&\lim_{\delta\to 0}
			\frac{1}{\delta}\int_{\gamma(x)}^{\gamma(x)+\delta}
			\ln{\tilde{u}(y^{\dag}(x)+ \tau w)} \dd \tau\\
			&=&\ln{\tilde{u}(y^{\dag}(x)+ \gamma(x) w)}.
		\end{eqnarray*}
		Namely, $\exp\{ w^{\top}\nabla f(x)\}=\tilde{u}(y^{\dag}(x)+
		\gamma(x) w)=\tilde{u}(x)$, which obviously satisfies (\ref{LyaS1}).
		
		(2) The second part contributes to verifying that the current $f(x)$
		satisfies the boundary condition of (\ref{BoundaryConditionS1}). We
		address this issue according to two different cases.
		
		\textit{Case I:} $\bar{\mathcal{C}}_{\bar{x}}=\emptyset$. In this case
		(\ref{BoundaryConditionS1}) is obviously true since it is ``$0=0$".
		
		\textit{Case II:} $\bar{\mathcal{C}}_{\bar{x}}\neq \emptyset$. In this
		case there include at least two complexes in
		{$\bar{\mathcal{C}}_{\bar{x}}$}, one acting as a reactant complex and the
		other as a resultant complex. According to the definition of
		$\bar{\mathcal{C}}_{\bar{x}}$ in (\ref{naive boundary complex set}), we
		have $v_{ji}=0$ and $v'_{ji'}=0$ for $\forall~j\in Z_{\bar{x}}$ if
		$v_{\cdot i},v'_{\cdot i'}\in\mathcal{C}_{\bar{x}}$. Note that
		$i=i'$ is possible. Further, from \textit{Lemma \ref{ZxPwNw}} we {get}
		$Z_{\bar{x}}\subseteq P_{w}$ or $Z_{\bar{x}}\subseteq N_{w}$. For
		simplicity, let $Z_{\bar{x}}\subseteq P_{w}$ for the following
		proof.
		
		Imitating the boundary condition of (\ref{BoundaryConditionS1}), for
		$\forall~ \bar{x}$ we define a function $h(x,a)$ from
		$\left(\{\bar{x}\}\cup(\bar{x}+\mathscr{S})\cap\mathbb{R}^{n}_{>0}\right)\times
		\mathbb{R}_{>0}$ to $\mathbb{R}$ as
		\begin{equation*}
		h(x,a)=\sum_{\{i|v_{\cdot
				i}\in\mathcal{C}_{\bar{x}}\}}k_{i}x^{v_{\cdot
				i}}-\sum_{\{i|v'_{\cdot
				i}\in\mathcal{C}_{\bar{x}}\}}k_{i}x^{v_{\cdot
				i}+m_{i}\tilde{w}}{a}^{m_{i}},
		\end{equation*}
		where $\tilde{w}\in\mathbb{R}^n$ with the $j$th ($j=1,\cdots,n$)
		entry satisfying $\tilde{w}_j=w_{j}$ if $j\in Z_{\bar{x}}$, and
		$\tilde{w}_j=0$ otherwise, and $m_i$ shares the same meaning with
		(\ref{LinearExpression}). For the first term in the right hand side,
		since for $\forall~j\in Z_{\bar{x}}$, $v_{ji}=0$, we have
		$\bar{x}^{v_{\cdot i}}\textgreater 0$. We further analyze the sign
		of $x^{v_{\cdot i}+m_{i}\tilde{w}}$ in the second term. At this time
		$v'_{\cdot i}\in\mathcal{C}_{\bar{x}}$, we thus have $v'_{ji}=0$ for
		$\forall~j\in Z_{\bar{x}}$, where $v'_{ji}=v_{ji}+m_{i}
		w_{j}~(\text{based on
			(\ref{LinearExpression})})=v_{ji}+m_{i}\tilde{w}_{j}=0$. This means
		$x^{v_{\cdot i}+m_{i}\tilde{w}}\textgreater 0$ and $m_i\textless 0$
		($j\in Z_{\bar{x}}\subseteq P_{w},\tilde{w}_j\textgreater
		0~\text{and}~v_{ji}\textgreater 0$) in the second term. Based on
		these facts, we get
		\begin{equation*}
		\lim_{a \to 0}h(x,a)=-\infty, ~~\lim_{a \to +\infty}h(x,a)>0~~\text{and}~~\frac{\partial}{\partial a}h(x,a)>0.
		\end{equation*}
		As a result, there must exist a unique positive function
		$\hat{a}(x)$ such that $h(x,\hat{a}(x))=0$ for $\forall~x$ by
		{the intermediate value theorem} and monotonicity. Further, the function
		$\hat{a}(x)$ is continuous differentiable by {the implicit function
		theorem}.
		
		Based on $h(x,a)$ we further define another continuous
		differentiable function $\tilde{h}(x,a)$ from
		$\left((\bar{x}+\mathscr{S})\cap\mathbb{R}^{n}_{>0}\right)\times
		\mathbb{R}_{>0}$ to $\mathbb{R}$ as
		\begin{equation*}
		\tilde{h}(x,a)=h(x,a)+\sum_{\{i|v_{\cdot
				i}\notin\mathcal{C}_{\bar{x}}\}}k_{i}x^{v_{\cdot i}}
		-\sum_{\{i|v'_{\cdot
				i}\notin\mathcal{C}_{\bar{x}}\}}k_{i}x^{v_{\cdot
				i}+m_{i}\tilde{w}}{a}^{m_{i}}.
		\end{equation*}
		For the second term in the right hand side, when $v_{\cdot i}\notin
		\mathcal{C}_{\bar{x}}$ there exists $j\in Z_{\bar{x}}$ such that
		$\bar{x}_j=0$ and $v_{ji}\textgreater 0$, i.e., $\bar{x}^{v_{\cdot
				i}}=0$. Analogously, for the third term when $v'_{\cdot i}\notin
		\mathcal{C}_{\bar{x}}$ there also exists $j\in Z_{\bar{x}}$ such
		that $\bar{x}_j=0$ and $v'_{ji}=v_{ji}+m_{i}
		w_{j}=v_{ji}+m_{i}\tilde{w}_{j}\textgreater 0$, i.e.,
		$\bar{x}^{v_{\cdot i}+m_{i}\tilde{w}}=0$. Hence, for
		$\forall~x\in(\bar{x}+\mathscr{S})\cap\mathbb{R}^{n}_{>0}$ we get
		\begin{equation*}
		\lim_{x\to\bar{x}} \sum_{\{i|v_{\cdot
				i}\notin\mathcal{C}_{\bar{x}}\}}k_{i}x^{v_{\cdot i}}=0
		~~\text{and}~~ \lim_{x\to\bar{x}} \sum_{\{i|v'_{\cdot
				i}\notin\mathcal{C}_{\bar{x}}\}}k_{i}x^{v_{\cdot
				i}+m_{i}\tilde{w}}=0.
		\end{equation*}
		Further we have
		$$\lim_{x\to\bar{x}}\frac{\partial}{\partial a}\tilde{h}(x,\hat{a}(x))=\frac{\partial}{\partial
			a}h(x,\hat{a}(x))\Big|_{x=\bar{x}}\textgreater 0,$$ which {means}
		$\frac{\partial}{\partial a}\tilde{h}(x,\hat{a}(x))\textgreater 0$ in
		$\mathcal{E}_1(\bar{x})$, a certain neighborhood of $\bar{x}$. Based
		on the same analysis, we can obtain $\frac{\partial^{2}}{\partial
			a^{2}}\tilde{h}(x,a)\leq 0$ in another certain neighborhood of
		$\bar{x}$, denoted by $\mathcal{E}_2(\bar{x})$.
		
		Suppose $v'_{\cdot p}\in\mathcal{C}_{\bar{x}}$ and consider the
		neighborhood of $\bar{x}$,
		$\mathcal{E}_1(\bar{x})\cap\mathcal{E}_2(\bar{x})$, within which let
		$$\theta_{1}(x)=-\left(\sum_{\{i|v_{\cdot
				i}\notin\mathcal{C}_{\bar{x}}\}}k_{i}x^{v_{\cdot
				i}}-\sum_{\{i|v'_{\cdot
				i}\notin\mathcal{C}_{\bar{x}}\}}k_{i}x^{v_{\cdot
				i}+m_{i}\tilde{w}}{\hat{a}(x)}^{m_{i}}\right) \left/\frac{\partial
			\tilde{h}(x,a)}{\partial a}\Big|_{a=\hat{a}(x)}\right.$$ and
		$$\theta_{2}(x)=\frac{\sum_{\{i|v'_{\cdot i}\notin\mathcal{C}_{\bar{x}}\}}\Big(\sum_{m_{i}>0}k_{i}x^{v_{\cdot i}+m_{i}\tilde{w}}[2 \hat{a}(x)]^{m_{i}}
			+\sum_{m_{i}<0}k_{i}x^{v_{\cdot
					i}+m_{i}\tilde{w}}\left[\frac{1}{2}\hat{a}(x)\right]^{m_{i}}\Big)}
		{k_{p}x^{v_{\cdot p}+m_{p}\tilde{w}}\hat{a}(x)^{m_{p}}},$$ then we
		get
		\begin{eqnarray*}
			\tilde{h}(x,\hat{a}(x)+\theta_{1}(x))\leq\tilde{h}(x,\hat{a}(x))+\frac{\partial
				\tilde{h}(x,a)}{\partial
				a}\Big|_{a=\hat{a}(x)}\theta_{1}(x)={h}(x,\hat{a}(x))=0.
		\end{eqnarray*}
		Further let $$\theta_{3}(x)=\text{max}\left\{ 1,~~\left|
		1-\frac{\sum_{\{i|v_{i}\notin\mathcal{C}_{\bar{x}}\}}k_{i}x^{v_{i}}}{k_{p}x^{v_{\cdot
					p}+m_{p}\tilde{w}}\hat{a}(x)^{m_{p}}}+\theta_{2}(x)\right|^{\frac{1}{m_p}}\right\}.$$
		It is easy to verify that $\lim_{x\to \bar{x}}\theta_{3}(x)=1$, from
		which we also obtain $1<\theta_{3}(x)<2$ when $x$ is in a certain
		neighborhood of $\bar{x}$, $\mathcal{E}_3(\bar{x})$. Hence, for
		$\forall~x\in\{\mathcal{E}_1(\bar{x})\cap\mathcal{E}_2(\bar{x})\cap\mathcal{E}_3(\bar{x})\}$
		we have $h(x,\theta_{3}(x)\hat{a}(x))\textgreater h(x,\hat{a}(x))$,
		$0\textgreater k_{p}x^{v_{\cdot
				p}+m_{p}\tilde{w}}\left(\theta_{3}(x)\hat{a}(x)\right)^{m_{p}}-k_{p}x^{v_{\cdot
				p}+m_{p}\tilde{w}}\hat{a}(x)^{m_{p}}$ and
		\begin{eqnarray*}
			\sum_{\{i|v_{\cdot i}\notin\mathcal{C}_{\bar{x}}\}}k_{i}x^{v_{\cdot
					i}} -\sum_{\{i|v'_{\cdot
					i}\notin\mathcal{C}_{\bar{x}}\}}k_{i}x^{v_{\cdot
					i}+m_{i}\tilde{w}}{(\theta_{3}(x)\hat{a}(x))}^{m_{i}}>\sum_{\{i|v_{\cdot
					i}\notin\mathcal{C}_{\bar{x}}\}}k_{i}x^{v_{\cdot
					i}}~~~~~~~~~~~~\\-\sum_{\{i|v'_{\cdot
					i}\notin\mathcal{C}_{\bar{x}}\}}\left(\sum_{m_{i}>0}k_{i}x^{v_{\cdot
					i}+m_{i}\tilde{w}}{(2\hat{a}(x))}^{m_{i}}
			+\sum_{m_{i}<0}k_{i}x^{v_{\cdot
					i}+m_{i}\tilde{w}}{\left(\frac{1}{2}\hat{a}(x)\right)}^{m_{i}}\right).
		\end{eqnarray*}
		Note that the sum of the left terms of these three inequalities is
		just expressed as
		$$\tilde{h}(x,\theta_{3}(x)\hat{a}(x))=h(x,\theta_{3}(x)\hat{a}(x))+\sum_{\{i|v_{\cdot
				i}\notin\mathcal{C}_{\bar{x}}\}}k_{i}x^{v_{\cdot i}}
		-\sum_{\{i|v'_{\cdot
				i}\notin\mathcal{C}_{\bar{x}}\}}k_{i}x^{v_{\cdot
				i}+m_{i}\tilde{w}}{(\theta_{3}(x)\hat{a}(x))}^{m_{i}}$$ while the
		sum of the right terms is greater than or equal to zero, i.e.,
		$$\tilde{h}(x,\theta_{3}(x)\hat{a}(x))\geq 0.$$
		Therefore, there {must exist} a solution between
		$\hat{a}(x)+\theta_{1}(x)$ and $\theta_{3}(x)\hat{a}(x)$, denoted by
		$\tilde{a}(x)$, such that $\tilde{h}(x,\tilde{a}(x))=0$, i.e.,
		$$h(x,\tilde{a}(x))+\sum_{\{i|v_{\cdot
				i}\notin\mathcal{C}_{\bar{x}}\}}k_{i}x^{v_{\cdot i}}
		-\sum_{\{i|v'_{\cdot
				i}\notin\mathcal{C}_{\bar{x}}\}}k_{i}x^{v_{\cdot
				i}+m_{i}\tilde{w}}(\tilde{a}(x))^{m_i}=0.$$ This together with the
		definition of $\tilde{h}(x,a)$ evaluated at $a=\tilde{a}(x)$ leads to
		$$\sum_{i=1}^{r} k_{i}x^{v_{\cdot i}}-\sum_{i=1}^{r} k_{i}x^{v_{\cdot
				i}+m_{i}\tilde{w}} (\tilde{a}(x))^{m_i}=0.$$ Comparing it to the
		Lyapunov PDE of (\ref{LyaS1}) yields that
		$u(x)=x^{\tilde{w}}\tilde{a}(x)$ is a solution of $(u-1)g(x,u)=0$.
		Note the facts that
		{
			$$\lim_{x\to
				\bar{x}}(\hat{a}(x)+\theta_1(x))=\lim_{x\to
				\bar{x}}(\theta_3(x)\hat{a}(x))=\lim_{x\to
				\bar{x}}\hat{a}(x),$$
			and $\lim_{x\to
				\bar{x}}\tilde{a}(x)$ lies between $\lim_{x\to
				\bar{x}}(\hat{a}(x)+\theta_1(x))$ and $\lim_{x\to
				\bar{x}}(\theta_3(x)\hat{a}(x))$, by {the squeeze
			theorem} we thus have
			$$\lim_{x\to\bar{x}}\tilde{a}(x)=\hat{a}(\bar{x})~\text{and}~\lim_{x\to\bar{x}}x^{\tilde{w}}\tilde{a}(x)=0.$$
			Further, there
			exists a certain neighborhood of $\bar{x}$,
			$\mathcal{E}_4(\bar{x})$, such that $x^{\tilde{w}}\tilde{a}(x)\neq
			1$. Hence, if
			$x\in\mathcal{E}_1(\bar{x})\cap\mathcal{E}_2(\bar{x})\cap\mathcal{E}_3(\bar{x})\cap\mathcal{E}_4(\bar{x})$,
			we have $g(x,x^{\tilde{w}}\tilde{a}(x))=0$, i.e.,
			$\tilde{u}(x)=x^{\tilde{w}}\tilde{a}(x)$.
		}
		
		Utilizing the above analysis, we {get}
		\begin{equation*}
		\text{L.H.S of Eq. (\ref{BoundaryConditionS1})}=\lim_{
			\begin{tiny}
			\begin{array}{c}
			x \to \bar{x}\\
			x\in \{\bar{x}+\mathscr{S}\}\cap\mathbb{R}^{n}_{>0}
			\end{array}
			\end{tiny}}
		h(x,\tilde{a}(x))=h(\bar{x},\hat{a}(\bar{x}))=0.
		\end{equation*}
		Therefore, $f(x)$ in the form of (\ref{SolutionOfS1}) satisfies the
		boundary condition of (\ref{BoundaryConditionS1}).
		
		Similarly, {in the case of} $Z_{x} \subseteq N_{w}$ the results hold too,
		which completes the proof.
	\end{proof}
	
	The following task focuses on verifying if $f(x)=
	\int_{0}^{\gamma(x)} \ln{\tilde{u}(y^{\dag}(x)+ \tau w)} \dd \tau$
	is able to serve as an Lyapunov function for MASs with
	$\text{dim}\mathscr{S}=1$.
	
	\begin{theorem}{\label{StabilityS1}}
		For a MAS {$(\mathcal{S},\mathcal{C},\mathcal{R},\mathcal{K})$} with
		$\text{dim}\mathscr{S}=1$ and a positive steady state
		$x^*\in\mathbb{R}^n_{\textgreater 0}$, let
		$\bar{\mathcal{C}}_{\bar{x}}$ defined by (\ref{naive boundary
			complex set}) represent the boundary complex set where $\bar{x}$ is
		any boundary point of any positive stoichiometric compatibility
		class induced by $\mathscr{S}$. If
		\begin{itemize}\itemindent 30pt
			\item $\bar{\mathcal{C}}_{\bar{x}}=\emptyset$ or
			$\bar{\mathcal{C}}_{\bar{x}}$ includes at least a reactant complex and a resultant complex;
			\item $w^\top\frac{\partial }{\partial x} g(x^*,1)\textless 0$ with
			$g(x,u)$ defined by (\ref{gDefinition}),
		\end{itemize}
		then the Lyapunov Function PDEs (\ref{LyaS1}) and (\ref{BoundaryConditionS1}) are qualified to generate a Lyapunov function \eqref{SolutionOfS1} to render this MAS to be locally asymptotically stable at
		$x^{*}$.
	\end{theorem}
	
	\begin{proof}
		Since $g_{x}(x,u)$ and $\tilde{u}(x)$ are continuous and $\tilde{u}(x^{*})=e^{w^\top\nabla
			f(x^*)}=1$, by the second condition listed in the theorem, there is a neighborhood of $x^{*}$, denoted as $\delta(x^{*})$, such that $\forall~ x\in\delta(x^{*})$ we have
		\begin{equation}{\label{eq. convex S1}}
		w^\top\frac{\partial }{\partial x} g(x,\tilde{u}(x))< 0.
		\end{equation}
		Moreover, for the function $f(x)$ given in \eqref{SolutionOfS1}, since $w^\top\nabla f(x)=\ln \tilde{u}(x)$, we have
		$$\nabla^2 f(x)w=\frac{\nabla
			\tilde{u}(x)}{\tilde{u}(x)}.$$
		Therefore, $\forall~ x\in\delta(x^{*})$ and $\forall~ \mu \in \mathscr{S}$ there is
		\begin{eqnarray}
		\mu^{\top} \nabla^{2}f(x) \mu
		&=& (\mu^{\top}w)^{2} \cdot w^{\top} \nabla^{2}f(x) w \notag\\
		&=& (\mu^{\top}w)^{2} \cdot
		\frac{w^{\top}\nabla
			\tilde{u}(x)}{\tilde{u}(x)} \notag\\
		&=&(\mu^{\top}w)^{2} \cdot
		\frac{-w^\top\frac{\partial }{\partial x}
			g(x,\tilde{u})
			/\frac{\partial}{\partial u} g(x,\tilde{u})
		}
		{\tilde{u}(x)} \notag\\
		&\geq& 0, \notag
		\end{eqnarray}
		where the last inequality follows from \eqref{eq. convex S1} and the equality holds if and only if $\mu=\mathbbold{0}_{n}$.
		Thus the condition (\ref{eq. strictly convex condition}) is satisfied.
		\textit{Theorem \ref{f(x)ExistenceS1}} has shown that the function \eqref{SolutionOfS1} is a solution of the PDEs, therefore the result holds immediately from \textit{Theorem \ref{thm lya function}}.
	\end{proof}
	
	\begin{remark}
		The condition $w^\top\frac{\partial }{\partial x} g(x^*,1)\textless
		0$ essentially characterizes some behaviors of the MAS system after
		linearization. From the dynamic equation (\ref{DeterministicModel})
		{in the case of} $\text{dim}\mathscr{S}=1$
		$$\dot{x}(t)=\sum_{i=1}^{r} k_{i} x^{v_{\cdot i}} (v'_{\cdot i}-v_{\cdot i})=\sum_{i=1}^{r} k_{i} x^{v_{\cdot i}}m_iw,$$
		we get the linearized form at $x=x^*$ as
		$$\dot{x}(t)=w\left(\frac{\partial g(x^*,1)}{\partial x}\right)^\top(x-x^*).$$
		It is clear that the coefficient matrix $w\left(\frac{\partial
			g(x^*,1)}{\partial x}\right)^\top$ is of rank one and thus has only
		two eigenvalues $w^\top\frac{\partial }{\partial x} g(x^*,1)$ plus
		$0$ if $n\neq 1$. Therefore, the condition $w^\top\frac{\partial
		}{\partial x} g(x^*,1)\textless 0$ means that the coefficient matrix
		need have a negative eigenvalue, which, namely, requests the
		linearized system of the MAS to be necessarily stable (but not
		necessarily asymptotically stable unless $n=1$) at $x=x^*$.
	\end{remark}
	
	{Another point should be noted that the solution \mbox{\eqref{SolutionOfS1}} of the Lyapunov Function PDEs has the similar form with the Lyapunov function constructed in Anderson and his coworkers' paper for Birth-Death processes\mbox{\cite{Anderson2015}}. Both functions are established by integrating a logarithmic function. The possible reasons are that the birth-death process studied in\mbox{\cite{Anderson2015}} is also a $1$-dimensional CRN, and that the Lyapunov function PDEs and the scaling limit of the non-equilibrium potential have the same origin. This phenomenon conversely implies that the PDEs can work for Birth-Death processes.}
	
	We further demonstrate the {efficiency} %efficacy
	of Lyapunov Function PDEs for CRNs
	with $\text{dim}\mathscr{S}=1$ through two examples.
	
	\begin{example}
		For the MAS
		\begin{eqnarray}
		S_{1}\stackrel{k_1}{\longrightarrow} S_{2}, \notag \qquad		2S_{2}\stackrel{k_2}{\longrightarrow} 2S_{1}, \notag
		\end{eqnarray}
	    {we have the species set $\mathcal{S}=\{S_1,S_2\}$, the complex set $\mathcal{C}=\{v_{\cdot 1},v'_{\cdot 1},v_{\cdot 2},v'_{\cdot 2}\}$, the reaction set $\mathcal{R}=(v_{\cdot 1}\to v'_{\cdot 1},v_{\cdot 2}\to v'_{\cdot 2})$, and the kinetics set $\mathcal{K}=(k_1,k_2)$, where  }
		\begin{equation*}
		v_{\cdot 1}=
		\left(
		\begin{array}{c}
		1\\ 0
		\end{array}
		\right),~~
		v'_{\cdot 1}=
		\left(
		\begin{array}{c}
		0\\ 1
		\end{array}
		\right),~~
		v_{\cdot 2}=
		\left(
		\begin{array}{c}
		0\\2
		\end{array}
		\right),~~ v'_{\cdot 2}= \left(
		\begin{array}{c}
		2\\ 0
		\end{array}
		\right),~~ \mathrm{dim}\mathscr{S}=1.
		\end{equation*}
		{By the mass-action kinetics, the dynamics of the system is expressed as}
		\begin{equation*}
		\left\{
		\begin{array}{lr}
		\dot{x}_1(t)=-k_1x_1+2k_2x_2^2,\\
		\dot{x}_2(t)=k_1x_1-2k_2x_2^2.
		\end{array}
		\right.
		\end{equation*}
		
		{By choosing} $w=(-1,1)^{\top}$ as {the basis} for $\mathscr{S}$, {we can write the Lyapunov Function PDEs in the form of \mbox{\eqref{LyaS1}}} where
		$m_{1}=1$ and $m_{2}=-2$. Moreover, from \textit{Proposition \ref{uExistence}} we
		have
		$$P_w=\{2\},~~N_w=\{1\}~~\mathrm{and}~~g(x,u)=k_1x_1-k_2x_2^2u^{-1}-k_2x_2^2u^{-2}.$$
		Furthermore, utilizing \textit{Lemma \ref{CoordinateTransformation}} and setting $g(x,u)=0$ we get {auxiliary functions}
		\begin{equation*}
		y^\dag(x)=\frac{1}{2}\left(
		\begin{array}{c}
		\mathbbm{1}_2^\top\\
		\mathbbm{1}_2^\top \\
		\end{array}
		\right)x,~~\gamma(x)=\frac{w^\top
			x}{2}~~\mathrm{and}~~\tilde{u}(x)=\frac{k_2x_2^2+x_2\sqrt{k_2^2x_2^2+4k_1k_2x_1}}{2k_1x_1}.
		\end{equation*}
		For the considered MAS, there include two types of boundary points, {one type of points are} $\bar{x}=(\bar{x}_{1},0)^{\top}$ with $\bar{x}_{1}>0$, {the other type of points are}
		$\bar{x}=(0,\bar{x}_{2})^{\top}$ with $\bar{x}_{2}>0$. Following the
		definition of $\bar{C}_{\bar{x}}$ in (\ref{naive boundary complex
			set}), we set
		\begin{equation*}
		\bar{C}_{\bar{x}}= \left\{
		\begin{array}{ll}
		\left\{v_{\cdot 1},v'_{\cdot 2}\right\}, & \bar{x}=(\bar{x}_{1},0)^{\top}~\mathrm{with}~\bar{x}_{1}>0;\\
		\left\{v_{\cdot 2},v'_{\cdot 1}\right\}, &
		\bar{x}=(0,\bar{x}_{2})^{\top}~\mathrm{with}~\bar{x}_{2}>0.
		\end{array}
		\right.
		\end{equation*}
		Finally, we obtain a solution, based on \textit{Theorem \ref{f(x)ExistenceS1}}, as
		$$f(x)= \int_{0}^{\gamma(x)}
		\ln{\frac{k_2[y^\dag(x)+\tau w]^{v_{\cdot
						2}}+\sqrt{k_2^2[y^\dag(x)+\tau w]^{2v_{\cdot
							2}}+4k_1k_2[y^\dag(x)+\tau w]^{v_{\cdot 1}+v_{\cdot
							2}}}}{2k_1[y^\dag(x)+\tau w]^{v_{\cdot 1}}}} \dd \tau$$ for the
		Lyapunov Function PDEs (\ref{LyaS1}) plus
		(\ref{BoundaryConditionS1}).
		
		Let $x^*\in\mathbb{R}^2_{\textgreater 0}$ be an equilibrium
		in $(x^*+\mathscr{S})\cap \mathbb{R}^2_{\textgreater 0}$. According
		to \textit{Theorem \ref{StabilityS1}}, since $w^{\top}\frac{\partial
			g(x^{*},1)}{\partial x}=-k_{1}-4k_{2}x^{*}_{2}<0$, the current
		$f(x)$ is an available Lyapunov function for suggesting the studied
		system to be locally asymptotically stable at $x^*$.
	\end{example}
	
	Note that the condition $w^{\top}\frac{\partial g(x^{*},1)}{\partial
		x}=-k_{1}-4k_{2}x^{*}_{2}\textless 0$ in \textit{Eaxmple 2} is
	always true, which in turn means it reasonable to set the condition
	of $w^{\top}\frac{\partial g(x^{*},1)}{\partial x}\textless 0$ in
	\textit{Theorem \ref{StabilityS1}}.

\begin{example}
{This $1$-dimensional MAS only contains a single species $S_1$ and has a reversible reaction structure, given by }
    	\begin{equation*}
		0 \autorightleftharpoons{$k_{1}$}{$k_{2}$} S_1, \quad 2S_1\autorightleftharpoons{$k_{3}$}{$k_{4}$}3S_1.
		\end{equation*}
{Denote $\mathcal{R}_i$ by the reaction with the rate coefficient $k_i$, $i=1,\cdots,4$, then $v_{\cdot 1}=v'_{\cdot 2}=0$, $v_{\cdot 2}=v'_{\cdot 1}=1$, $v_{\cdot 3}=v'_{\cdot 4}=2$, and $v_{\cdot 4}=v'_{\cdot 3}=3$. Further by setting $k_1=k_4=2$ and $k_2=k_3=1$, the dynamics of this MAS is written as}
		\begin{equation*}
			\dot{x}_{1}(t)=2-x_1+x_1^{2}-2x_1^{3}.
		\end{equation*}
{Clearly, the system admits a unique equilibrium point $x^*_1=1$. Note that this equilibrium is not complex balanced since at it the zero complex does not balance between the reaction rate $2$ and the production rate $1$. The pseudo-Helmholtz free energy function is thus not an appropriate Lyapunov function for stability analysis. Instead, we use the current Lyapunov Function PDEs for $1$-dimensional CRNs to produce the Lyapunov function, i.e., Eq. \mbox{(\ref{SolutionOfS1})}. We set $w=1$ as the base for the stoichiometric subspace, then $m_1=m_3=1$, $m_2=m_4=-1$ and $g(x,u)=2+x^2-\frac{1}{u}(x+2x^3)$. Further, we get}
	   \begin{equation*}
		y^\dag(x)=1,
		~~\gamma(x)=x-1,
		~~\tilde{u}(x)=\frac{x+2x^{3}}{2+x^{2}}~~\mathrm{and}
		~~\bar{C}_{\bar{x}}=\{0\}.
		\end{equation*}
{Finally, based on on \textit{Theorem} \mbox{\ref{f(x)ExistenceS1}}, the solution for the Lyapunov Function PDEs \mbox{(\ref{LyaS1})} plus	\mbox{(\ref{BoundaryConditionS1})} is expressed as}
		$$f(x)= \int_0^{x-1} \ln\left(\frac{(1+\tau)+2(1+\tau)^{3}}{2+(1+\tau)^{2}}\right) d \tau.$$
{Since $w^{\top}\frac{\partial g(x^{*},1)}{\partial x}=2x^{*} -1 -6 (x^{*})^{2}=-5\leq 0$, the above function $f(x)$ is a valid Lyapunov function for suggesting the studied
		system to be locally asymptotically stable at $x^*=1$. }		
\end{example}
	
	\section{Lyapunov Function PDEs for some CRNS with $\text{dim}\mathscr{S}\geq
		2$} For general CRNs with $\text{dim}\mathscr{S}\geq 2$, we are not
	able to prove that the Lyapunov Function PDEs (\ref{LyaPDEs}) plus
	(\ref{BoundaryCondition}) work validly in this paper. However, they
	are shown {valid} for some special CRNs with
	$\text{dim}\mathscr{S}\geq 2$.
	
	\subsection{CRNs of $\text{dim}\mathscr{S}\geq
		2$ Composed of a Complex Balanced CRN and a series of CRNs of
		$\text{dim}\mathscr{S}=1$} Consider a MAS
	{$(\mathcal{S},\mathcal{C},\mathcal{R},\mathcal{K})$} as a combination of a
	complex balanced MAS, labeled as
	{$(\mathcal{S}^{(0)},\mathcal{C}^{(0)},\mathcal{R}^{(0)},\mathcal{K}^{(0)})$}, and a
	few MASs of $1$-dimensional stoichiometric subspace, denoted by
	{$(\mathcal{S}^{(p)},\mathcal{C}^{(p)},\mathcal{R}^{(p)},\mathcal{K}^{(p)})$}
	($p={1,\cdots,\ell}$), respectively. These sub-networks are assumed
	to be independent each other. Namely, for
	$\forall~p,q\in\{0,\cdots,\ell\}$, if $p\neq q$, then
	$\mathcal{S}^{(p)}\cap\mathcal{S}^{(q)}=\emptyset$. We define this
	class of CRNs as ``Com-$\ell$Sub$1$" CRNs, and the corresponding MASs are named ``Com-$\ell$Sub$1$" MASs.
	
	In every sub-network
	{$(\mathcal{S}^{(p)},\mathcal{C}^{(p)},\mathcal{R}^{(p)})$}, let
	$n_p$, $r_p$ represent the number of species and of reactions,
	$v_{\cdot i^{(p)}}$ the reactant complex and $v'_{\cdot i^{(p)}}$
	the resultant complex of the $i$th reaction $(i=1,\cdots,r_p)$,
	respectively. Also, denote
	$$n=\sum_{p=0}^\ell n_p,~v_{\cdot i}^{(p)}=\bigotimes_{q=0}^{p-1} \mathbbold{0}_{n_{q}}
	\bigotimes v_{\cdot i^{(p)}} \bigotimes_{q=p+1}^{\ell}
	\mathbbold{0}_{n_{q}}~\text{and}~{v'_{\cdot
			i}}^{(p)}=\bigotimes_{q=0}^{p-1} \mathbbold{0}_{n_{q}} \bigotimes
	v'_{\cdot i^{(p)}} \bigotimes_{q=p+1}^{\ell} \mathbbold{0}_{n_{q}},$$
	where $\bigotimes$ is the Cartesian product, and $n_q=0$ if
	$q\textless 0$ or $q\textgreater \ell$, then the MAS
	{$(\mathcal{S},\mathcal{C},\mathcal{R},\mathcal{K})$} under consideration is
	expressed as
	\begin{equation}{\label{CRN s=n}}
	\mathcal{S}=\bigcup_{p=0}^{\ell}\mathcal{S}^{(p)},~~
	\mathcal{C}=\bigcup_{p=0}^{\ell}\bigcup_{i=1}^{r_{p}}\{v_{\cdot
		i}^{(p)},{v'_{\cdot i}}^{(p)}\}~\text{and}~
	\mathcal{R}=\bigcup_{p=0}^{\ell}\bigcup_{i=1}^{r_{p}}\{v_{\cdot
		i}^{(p)}\stackrel{k_i^{(p)}}{\longrightarrow} {v'_{\cdot i}}^{(p)}\}
	\end{equation}
	with the dynamics to be
	\begin{equation}{\label{kinetic equation s=n}}
	\dot x = \sum_{p=0}^{\ell} \sum_{i=1}^{r_{p}} k_{i}^{(p)}
	x^{v_{\cdot i}^{(p)}} \left({v'_{\cdot i}}^{(p)}-v_{\cdot
		i}^{(p)}\right),
	\end{equation}
	where the state
	$x=\bigotimes_{p=0}^{\ell}x^{(p)}\in\mathbb{R}^n_{\geq 0}$ and
	$x^{(p)}\in\mathbb{R}^{n_p}_{\geq 0}$ is the state of the mass action system
	$\{\mathcal{S}^{(p)},\mathcal{C}^{(p)},\mathcal{R}^{(p)},\mathcal{K}^{(p)}\}$. Note
	that for $\forall~p,q\in\{0,\cdots,\ell\}$ if $p\neq q$ then
	$$\left\{\bigcup_{i=1}^{r_{p}}\{v_{\cdot
		i}^{(p)}\stackrel{k_i^{(p)}}{\longrightarrow} {v'_{\cdot
			i}}^{(p)}\}\right\}\bigcap\left\{\bigcup_{i=1}^{r_{q}}\{v_{\cdot
		i}^{(q)}\stackrel{k_i^{(q)}}{\longrightarrow} {v'_{\cdot
			i}}^{(q)}\}\right\}=\emptyset.$$ Therefore, the number of reactions
	contained in the MAS of (\ref{CRN s=n}) is $r=\sum_{p=0}^\ell r_p$.
	
	In the following, we will expound that the Lyapunov Function PDEs
	induced by Com-$\ell$Sub$1$ MASs also work validly for stability
	analysis by generating a solution as the Lyapunov function.
	
	\begin{lemma}{\label{StoichiometricSubspace s=n}}
		The stoichiometric subspace $\mathscr{S}$ of a Com-$\ell$Sub1 MAS
		satisfies
		\begin{equation}{\label{eq stoichiometic subspace s=n }}
		\mathscr{S}=\bigotimes_{p=0}^{\ell}\mathscr{S}^{(p)}~~\text{and}~~\mathrm{dim}\mathscr{S}=\sum_{p=0}^\ell\mathrm{dim}\mathscr{S}^{(p)}=\mathrm{dim}\mathscr{S}^{(0)}+\ell,
		\end{equation}
		where $\mathscr{S}^{(p)}$ is the stoichiometric subspace of
		{$(\mathcal{S}^{(p)},\mathcal{C}^{(p)},\mathcal{R}^{(p)},\mathcal{K}^{(p)})$}.
	\end{lemma}
	
	\begin{proof}
		Since
		\begin{eqnarray*}
			\mathscr{S}&=&\text{span}\left\{
			\bigcup_{p=0}^{\ell}\bigcup_{i=1}^{r_{p}}\{v_{\cdot
				i}^{(p)}-{v'_{\cdot i}}^{(p)}\} \right\}\\
			&=&
			\sum_{p=0}^{\ell}\text{span}\left\{\bigcup_{i=1}^{r_{p}}\{v_{\cdot
				i}^{(p)}-{v'_{\cdot i}}^{(p)}\} \right\}\\
			&=& \sum_{p=0}^{l} \left(
			\bigotimes_{q=0}^{p-1} \mathbbold{0}_{n_{q}} \bigotimes \mathscr{S}^{(p)}
			\bigotimes_{q=p+1}^{\ell} \mathbbold{0}_{n_{q}}
			\right)\\
			&=&\bigoplus_{p=0}^{l} \left(
			\bigotimes_{q=0}^{p-1} \mathbbold{0}_{n_{q}} \bigotimes \mathscr{S}^{(p)}
			\bigotimes_{q=p+1}^{\ell} \mathbbold{0}_{n_{q}}
			\right),
		\end{eqnarray*}
		we have $\mathscr{S}=\bigotimes_{p=0}^{\ell}\mathscr{S}^{(p)}$ and
		$\mathrm{dim}\mathscr{S}=\sum_{p=0}^\ell\mathrm{dim}\mathscr{S}^{(p)}=\mathrm{dim}\mathscr{S}^{(0)}+\ell$.
		Here, $\bigoplus$ is the direct sum.
	\end{proof}
	
	\begin{lemma}{\label{lem boundary point s=n}}
		For any state $x\in\mathbb{R}^n_{\geq 0}$ of a Com-$\ell$Sub1 mass action system,
		if $x\in\partial_{(x+\mathscr{S})\cap\mathbb{R}^n_{\textgreater
				0}}$ $({(x+\mathscr{S})\cap\mathbb{R}^n_{\textgreater
				0}\neq \emptyset})$, then for $\forall~p\in\{0,\cdots,\ell\}$ we have
		$x^{(p)}\in\mathbb{R}^n_{\textgreater 0}$ or
		$x^{(p)}\in\partial_{(x^{(p)}+\mathscr{S}^{(p)})\cap\mathbb{R}^n_{\textgreater
				0}}$. Furthermore, there exists at least one $q\in\{0,\cdots,\ell\}$
		such that
		$x^{(q)}\in\partial_{(x^{(q)}+\mathscr{S}^{(q)})\cap\mathbb{R}^n_{\textgreater
				0}}$.
	\end{lemma}
	
	\begin{proof}
		Since $(x+\mathscr{S})\cap\mathbb{R}^n_{\textgreater
			0}\neq\emptyset${, based} on \textit{Lemma \ref{StoichiometricSubspace s=n}}, we get
		\begin{eqnarray*}
			(x+\mathscr{S})\bigcap\mathbb{R}^{n}_{\textgreater 0}
			&=&\left(\bigotimes_{p=0}^{\ell} x^{(p)} + \bigotimes_{p=0}^{\ell}
			\mathscr{S}^{(p)} \right)\bigcap
			\left(\bigotimes_{p=0}^{\ell}\mathbb{R}^{n_{p}}_{\textgreater 0}\right)\\
			&=&\bigotimes_{p=0}^{\ell}\left[(x^{(p)}+\mathscr{S}^{(p)})\bigcap\mathbb{R}^{n_{p}}_{\textgreater
				0}\right]\neq\emptyset.
		\end{eqnarray*}
		Hence,
		$(x^{(p)}+\mathscr{S}^{(p)})\bigcap\mathbb{R}^{n_{p}}_{\textgreater
			0}\neq\emptyset$, i.e., $x^{(p)}\in\mathbb{R}^{n_p}_{\textgreater
			0}$ or
		$x^{(p)}\in\partial_{(x^{(p)}+\mathscr{S}^{(p)})\cap\mathbb{R}^n_{\textgreater
				0}}$ for $\forall~p$.
		
		Besides, if $x^{(p)}\in\mathbb{R}^{n_p}_{\textgreater 0}$ for
		$\forall~p$, then
		$x=\bigotimes_{p=0}^{\ell}x^{(p)}\in\mathbb{R}^{n}_{>0}$, which
		contradicts with the condition
		$x\in\partial_{(x+\mathscr{S})\cap\mathbb{R}^n_{\textgreater 0}}$.
		Thus, there exists at least one $q\in\{0,\cdots,\ell\}$ such that
		$x^{(q)}\in\partial_{(x^{(q)}+\mathscr{S}^{(q)})\cap\mathbb{R}^{n_q}_{\textgreater
				0}}$.
	\end{proof}
	
	\begin{corollary}\label{BoundaryPointGao}
		For any boundary point $\bar{x}\in
		\partial_{(\bar{x}+\mathscr{S})\bigcap\mathbb{R}^{n}_{\textgreater 0}}$ $({(\bar{x}+\mathscr{S})\cap\mathbb{R}^n_{\textgreater
				0}\neq \emptyset})$ of a
		Com-$\ell$Sub1 MAS, there exists an index set
		$\mathbb{P}_{\bar{x}}\subseteq \{0,\cdots,\ell\}$ such that if
		$p\in\mathbb{P}_{\bar{x}}$ then
		$x^{(p)}\in\partial_{(x^{(p)}+\mathscr{S}^{(p)})\cap\mathbb{R}^{n_q}_{\textgreater
				0}}$, which is denoted by $\bar{x}^{(p)}$ in the following.
	\end{corollary}
	
	\begin{lemma}{\label{lem naive boundary complex set s=n}}
		For a Com-$\ell$Sub1 MAS, let $\bar{x}$ represent any boundary point
		of any positive stoichiometric compatibility class, the naive
		boundary complex set of $\bar{x}$ is
		\begin{equation}{\label{naive boundary complex set s=n}}
		\bar{\mathcal{C}}_{\bar{x}}=\bigcup_{p\in \mathbb{P}_{\bar{x}}}
		\left\{
		\bigotimes_{q=0}^{p-1} \mathbbold{0}_{n_{q}} \bigotimes z
		\bigotimes_{q=p+1}^{\ell} \mathbbold{0}_{n_{q}}
		\big| z\in \bar{\mathcal{C}}_{\bar{x}^{(p)}}^{(p)}
		\right\}
		\bigcup_{p\notin \mathbb{P}_{\bar{x}}} \left\{
		\bigotimes_{q=0}^{p-1} \mathbbold{0}_{n_{q}} \bigotimes z
		\bigotimes_{q=p+1}^{\ell} \mathbbold{0}_{n_{q}}
		\big| z\in \mathcal{C}^{(p)}
		\right\},
		\end{equation}
		where $\bar{\mathcal{C}}_{\bar{x}^{(p)}}^{(p)}$ is the naive
		boundary complex set of $\bar{x}^{(p)}$ for {$(\mathcal{S}^{(p)},\mathcal{C}^{(p)},\mathcal{R}^{(p)},\mathcal{K}^{(p)})$}.
	\end{lemma}
	
	\begin{proof}
		According to the definition of the naive boundary complex set in
		(\ref{naive boundary complex set}), the one for the Com-$\ell$Sub1
		MAS is
		\begin{eqnarray*}
			\bar{\mathcal{C}}_{\bar{x}}&=& \bigcup_{p=0}^{\ell}\bigg[ \left\{
			v_{\cdot i}^{(p)}~|~\forall \{j\}_{1}^{n},\exists
			\epsilon\textgreater 0,\bar{x}_{j}\geq\epsilon v_{ji}^{(p)} \right\}
			\bigcup\left\{ {v'_{\cdot i}}^{(p)}~ |~\forall \{j\}_{1}^{n},\exists
			\epsilon\textgreater 0, \bar{x}_{j}\geq\epsilon {v'_{ji}}^{(p)}
			\right\}\bigg]\\
			&=& \bigcup_{p=0}^{\ell}\Bigg[ \left\{
			\bigotimes_{q=0}^{p-1}\mathbbold{0}_{n_q}\bigotimes v_{\cdot
				i^{(p)}}\bigotimes_{q=p+1}^{\ell}\mathbbold{0}_{n_q}~
			|~\forall\{j\}_{1}^{n_p},\exists \epsilon>0, x_{j}^{(p)}\geq\epsilon
			v_{{j i}^{(p)}}
			\right\}\\
			&&~~~~~~~\bigcup\left\{
			\bigotimes_{q=0}^{p-1}\mathbbold{0}_{n_q}\bigotimes v'_{\cdot
				i^{(p)}}\bigotimes_{q=p+1}^{\ell}\mathbbold{0}_{n_q}~
			|~\forall\{j\}_{1}^{n_p},\exists \epsilon>0, x_{j}^{(p)}\geq\epsilon
			{v'_{j i^{(p)}}}
			\right\}\Bigg]\\
			&=& \bigcup_{p\in \mathbb{P}_{\bar{x}}} \left\{
			\bigotimes_{q=0}^{p-1} \mathbbold{0}_{n_{q}} \bigotimes z
			\bigotimes_{q=p+1}^{\ell} \mathbbold{0}_{n_{q}}
			\big| z\in \bar{\mathcal{C}}_{\bar{x}^{(p)}}^{(p)}
			\right\}
			\bigcup_{p\notin \mathbb{P}_{\bar{x}}} \left\{
			\bigotimes_{q=0}^{p-1} \mathbbold{0}_{n_{q}} \bigotimes z
			\bigotimes_{q=p+1}^{\ell} \mathbbold{0}_{n_{q}}
			\big| z\in \mathcal{C}^{(p)}
			\right\}.
		\end{eqnarray*}
		Therefore, the result is true.
	\end{proof}
	
	\begin{lemma}{\label{lem LyaPDEs s=n}}
		For a Com-$\ell$Sub1 MAS, if the boundary complex set is chosen as
		the naive one given in Lemma \ref{lem naive boundary complex set s=n}, the Lyapunov Function PDEs are
		\begin{equation}{\label{LyaPDE s=n}}
		\sum_{p=0}^{\ell} \left( \sum_{i=1}^{r_{p}} \varTheta_1(i,p) -
		\sum_{i=1}^{r_{p}} \varTheta_2(i,p) \right) =0
		\end{equation}
		and
		\begin{eqnarray}{\label{BoundaryCondition s=n}}
		&&\sum_{p\in \mathbb{P}_{\bar{x}}} \lim_{
			\begin{tiny}
			\begin{array}{c}
			x^{(p)} \to \bar{x}^{(p)}\\
			x^{(p)}\in
			(\bar{x}^{(p)}+\mathscr{S}^{(p)})\cap\mathbb{R}^{n_{p}}_{\textgreater
				0}
			\end{array}
			\end{tiny}
		}  \Bigg( \sum_{\{i|v_{\cdot
				i^{(p)}}\in\bar{\mathcal{C}}_{\bar{x}^{(p)}}^{(p)}\}}
		\varTheta_1(i,p)- \sum_{\{i|v'_{\cdot
				i^{(p)}}\in\bar{\mathcal{C}}_{\bar{x}^{(p)}}^{(p)}\}}
		\varTheta_2(i,p) \Bigg)\\
		&+&\sum_{p\notin \mathbb{P}_{\bar{x}}} \Bigg( \sum_{i=1}^{r_{p}}
		\varTheta_1(i,p) - \sum_{i=1}^{r_{p}}\varTheta_2(i,p) \Bigg)=0,\notag
		\end{eqnarray}
		where $$\varTheta_1(i,p)=k_{i}^{(p)}{x^{(p)}}^{v_{\cdot
				i^{(p)}}},~~~~~ \varTheta_2(i,p)=k_{i}^{(p)}{x^{(p)}}^{v_{\cdot
				i^{(p)}}} \exp\left\{(v'_{\cdot i^{(p)}}-v_{\cdot i^{(p)}})^{\top}
		\frac{\partial f(x)}{\partial x^{(p)}} \right\},$$
		$x=\bigotimes_{p=0}^{\ell} x^{(p)} \in \mathbb{R}^{n}_{>0}$, and
		$\bar{x}$ represents any boundary point of any positive
		stoichiometric compatibility class.
	\end{lemma}
	
	\begin{proof}
		Referring to the Lyapunov Function PDE of (\ref{LyaPDEs}), we can
		write the version for the case of a Com-$\ell$Sub1 MAS to be
		\begin{equation*}
		\sum_{p=0}^{\ell} \left( \sum_{i=1}^{r_{p}} k_{i}^{(p)} x^{v_{\cdot
				i}^{(p)}} - \sum_{i=1}^{r_{p}} k_{i}^{(p)} x^{v_{\cdot i}^{(p)}}
		\exp\left\{({v'_{\cdot i}}^{(p)}-v_{\cdot i}^{(p)})^{\top} \nabla
		f(x) \right\} \right) =0.
		\end{equation*}
		Since $x=\bigotimes_{q=0}^{\ell}
		x^{(q)}\in\mathbb{R}^n_{\textgreater 0}$ and $v_{\cdot
			i}^{(p)}=\bigotimes_{q=0}^{p-1} \mathbbold{0}_{n_{q}} \bigotimes
		v_{\cdot i^{(p)}} \bigotimes_{q=p+1}^{\ell} \mathbbold{0}_{n_{q}}$,
		we have
		\begin{equation*}
		x^{v_{\cdot i}^{(p)}} = \left(\bigotimes_{q=0}^{\ell}x^{(q)} \right)
		^{\bigotimes_{q=0}^{p-1}\mathbbold{0}_{n_{q}}\bigotimes v_{\cdot
				i^{(p)}}\bigotimes_{q=p+1}^{\ell}\mathbbold{0}_{n_{q}}}
		={x^{(p)}}^{v_{\cdot i^{(p)}}}.
		\end{equation*}
		Also, we have
		\begin{eqnarray*}
			({v'_{\cdot i}}^{(p)}-v_{\cdot i}^{(p)})^{\top} \nabla f(x) &=&
			\left(\bigotimes_{q=0}^{p-1}\mathbbold{0}_{n_{q}}
			\bigotimes\left(v'_{\cdot i^{(p)}}-v_{\cdot
				i^{(p)}}\right)^{\top}
			\bigotimes_{q=p+1}^{\ell}\mathbbold{0}_{n_{q}}
			\right)^{\top}
			\left(
			\bigotimes_{q=0}^{\ell}\frac{\partial f(x)}{\partial x^{(p)}}
			\right)\\
			&=&\left(v'_{\cdot i^{(p)}}-v_{\cdot i^{(p)}}\right)^{\top}
			{\frac{\partial f(x)}{\partial x^{(p)}}}.
		\end{eqnarray*}
		Hence, we get (\ref{LyaPDE s=n}).
		
		By choosing (\ref{naive boundary complex set s=n}) as the boundary
		complex set and referring to (\ref{BoundaryCondition}), we can write
		the boundary condition for the Lyapunov function PDE (\ref{LyaPDE
			s=n}) to be (\ref{BoundaryCondition s=n}).
	\end{proof}
	
	Clearly, the Lyapunov Function PDEs (\ref{LyaPDE s=n}) and
	(\ref{BoundaryCondition s=n}) for a Com-$\ell$Sub1 MAS are a
	combination of those PDEs of all sub-systems
	{$(\mathcal{S}^{(p)},\mathcal{C}^{(p)},\mathcal{R}^{(p)},\mathcal{K}^{(p)})_{p=0}^\ell$}.
	An immediate idea is to set a solution of the current Lyapunov
	Function PDEs also as a combination of the solutions obtained from
	the Lyapunov Function PDEs of all sub-systems.
	
	\begin{theorem}{\label{thm solution s=n}}
		For a Com-$\ell$Sub1 mass action system, the sub-system
		{$(\mathcal{S}^{(0)},\mathcal{C}^{(0)},\mathcal{R}^{(0)},\mathcal{K}^{(0)})$} is
		assumed to admit a positive complex balanced equilibrium while every
		other sub-system
		{$(\mathcal{S}^{(p)},\mathcal{C}^{(p)},\mathcal{R}^{(p)},\mathcal{K}^{(p)})$}
		($p={1,\cdots,\ell}$) is supposed to have a positive equilibrium.
		Further let every sub-network from $p=1$ to $\ell$ possess the naive
		boundary complex set $\bar{\mathcal{C}}_{\bar{x}^{(p)}}^{(p)}$ as
		the respective boundary complex set, and also,
		$\bar{\mathcal{C}}_{\bar{x}^{(p)}}^{(p)}=\emptyset$ or $\bar{\mathcal{C}}_{\bar{x}^{(p)}}^{(p)}$ includes at least a reactant
		complex and a resultant complex. Then the Lyapunov Function PDEs
		(\ref{LyaPDE s=n}) and (\ref{BoundaryCondition s=n}) for this
		Com-$\ell$Sub1 MAS admit a twice continuous differentiable solution in the form of
		\begin{equation}{\label{solution s=n}}
		f(x)=\sum_{p=0}^{\ell} f_{p}(x^{(p)}),
		\end{equation}
		where $x=\bigotimes_{p=0}^{\ell}
		x^{(p)}\in\mathbb{R}^n_{\textgreater 0}$, and $f_{p}(x^{(p)})$ is a
		solution defined by (\ref{Gibbs}) in case of $p=0$ and
		by (\ref{SolutionOfS1}) in case of others $p$ for the Lyapunov Function
		PDEs of every sub-network
		{$(\mathcal{S}^{(p)},\mathcal{C}^{(p)},\mathcal{R}^{(p)},\mathcal{K}^{(p)})$.}
	\end{theorem}
	
	\begin{proof}
		According to \textit{Theorems \ref{thm solution of complex balanced crn}} and \textit{\ref{f(x)ExistenceS1}}, under the known
		conditions the Lyapunov Function PDEs of every sub-system included
		in the Com-$\ell$Sub1 MAS have a twice continuous differentiable solution defined by (\ref{Gibbs})
		in case of $p=0$ and (\ref{SolutionOfS1}) in case of others $p$ in
		the area $\mathbb{R}^{n_p}_{\textgreater 0}$. Denote these solutions
		by $f_p(x^{(p)})$ from $p=0$ to $\ell$ respectively, and further
		substitute each one into the corresponding Lyapunov Function PDEs,
		then we get
		\begin{equation*}
		\sum_{i=1}^{r_{p}} k_{i}^{(p)}{x^{(p)}}^{v_{\cdot i^{(p)}}} -
		\sum_{i=1}^{r_{p}} k_{i}^{(p)}{x^{(p)}}^{v_{\cdot i^{(p)}}}
		\exp\left\{({v'_{\cdot i^{(p)}}}-v_{\cdot i^{(p)}})^{\top} \nabla
		f_p(x^{(p)}) \right\} =0.
		\end{equation*}
		Combining the sum of these equations from $p=0$ to $\ell$ and the
		fact that $f(x)=\sum_{p=0}^{\ell} f_{p}(x^{(p)})$ leads to
		$\frac{\partial f(x)}{\partial x^{(p)}}=\nabla f_{p}(x^{(p)})$ will
		yield the current $f(x)$ with $x=\bigotimes_{p=0}^{\ell}
		x^{(p)}\in\mathbb{R}^n_{\textgreater 0}$ satisfying the Lyapunov
		Function PDE of (\ref{LyaPDE s=n}).
		
		Consider any boundary point $\bar{x}$ of any positive stoichiometric
		compatibility class for this network system. According to \textit{Corollary
			\ref{BoundaryPointGao}}, there exists a nonempty $\mathbb{P}_{\bar{x}}\subseteq
		\{0,\cdots,\ell\}$ so that when $p\in \mathbb{P}_{\bar{x}}$ the
		$p$th entry of $\bar{x}$ is a boundary point of a certain positive
		stoichiometric compatibility class of the sub-system
		{$(\mathcal{S}^{(p)},\mathcal{C}^{(p)},\mathcal{R}^{(p)},\mathcal{K}^{(p)})$}. For
		those $p\in \mathbb{P}_{\bar{x}}$ since $f_{p}(x^{(p)})$ satisfies
		the boundary condition (\ref{BoundaryConditionS1}), we have
		\begin{eqnarray*}
			&&\lim_{
				\begin{tiny}
					\begin{array}{c}
						x^{(p)} \to \bar{x}^{(p)}\\
						x^{(p)}\in
						(\bar{x}^{(p)}+\mathscr{S}^{(p)})\cap\mathbb{R}^{n_{p}}_{\textgreater
							0}
					\end{array}
				\end{tiny}
			}\sum_{\{i|v_{\cdot
					i^{(p)}}\in\bar{\mathcal{C}}_{\bar{x}^{(p)}}^{(p)}\}}\varTheta_1(i,p)-
			\sum_{\{i|v'_{\cdot
					i^{(p)}}\in\bar{\mathcal{C}}_{\bar{x}^{(p)}}^{(p)}\}}\varTheta_2(i,p)
			=0.
		\end{eqnarray*}
		Namely, the first term in the left hand of (\ref{BoundaryCondition
			s=n}) is equal to $0$. The second term is also equal to $0$ since
		for those $p\notin \mathbb{P}_{\bar{x}}$ every $f_{p}(x^{(p)})$
		supports the Lyapuonv function (\ref{LyaPDEs}). This completes the
		proof.
	\end{proof}
	
	\begin{lemma}\label{Gaolemma1}
		A state $x^*=\bigotimes_{p=0}^{\ell}
		{x^{(p)}}^{*}\in\mathbb{R}^n_{\textgreater 0}$ is a positive
		equilibrium of a Com-$\ell$Sub1 MAS if and only if for any
		$p=1,\cdots,\ell$, ${x^{(p)}}^{*}\in\mathbb{R}^{n_p}_{\textgreater
			0}$ is a positive equilibrium of the sub-system
		{$(\mathcal{S}^{(p)},\mathcal{C}^{(p)},\mathcal{R}^{(p)},\mathcal{K}^{(p)})$}.
	\end{lemma}
	
	\begin{proof}
		The result is immediate by inserting the state
		$x^*=\bigotimes_{p=0}^{\ell}
		{x^{(p)}}^{*}\in\mathbb{R}^n_{\textgreater 0}$ into the dynamics of
		the Com-$\ell$Sub1 MAS (\ref{kinetic equation s=n}).
	\end{proof}
	
	\begin{theorem}{\label{thm stable s=n}}
		For any $\ell\geq 1$, consider a Com-$\ell$Sub1 mass action system with the
		sub-system
		{$(\mathcal{S}^{(0)},\mathcal{C}^{(0)},\mathcal{R}^{(0)},\mathcal{K}^{(0)})$}
		admitting a complex balanced equilibrium
		${x^{(0)}}^*\in\mathbb{R}^{n_0}_{\textgreater 0}$ and other
		sub-systems
		{$(\mathcal{S}^{(p)},\mathcal{C}^{(p)},\mathcal{R}^{(p)},\mathcal{K}^{(p)})_{p=1}^{\ell}$}
		respectively admitting an equilibrium
		${x^{(p)}}^*\in\mathbb{R}^{n_p}_{\textgreater 0}$. Also, for any
		sub-system
		{$(\mathcal{S}^{(p)},\mathcal{C}^{(p)},\mathcal{R}^{(p)},\mathcal{K}^{(p)})~(p\in\{1,\cdots,\ell\})$}
		the boundary complex set is chosen as the naive boundary complex set
		$\bar{\mathcal{C}}_{\bar{x}^{(p)}}^{(p)}$ defined by (\ref{naive
			boundary complex set}), and
		$\bar{\mathcal{C}}_{\bar{x}^{(p)}}^{(p)}=\emptyset$ or $\bar{\mathcal{C}}_{\bar{x}^{(p)}}^{(p)}$ includes both reactant complexes and resultant complexes. If for all
		$\{p\}_{1}^{\ell}$ the conditions
		\begin{equation*}
		w_p^\top\frac{\partial }{\partial x^{(p)}}
		g_{p}(x^{(p)},1)\Big|_{x^{(p)}={x^{(p)}}^*}\textless 0
		\end{equation*}
		are true, then the Lyapunov Function PDEs induced by this
		Com-$\ell$Sub1 MAS are able to produce a solution \eqref{solution s=n} as a Lyapunov
		function serving for analyzing the local asymptotic stability of
		the network system. Here,
		$w_p\in\mathbb{R}^{n_p}\backslash\{\mathbbold{0}_{n_p}\}$ is a set
		of bases of $\mathscr{S}^{(p)}$ and $g_{p}(x^{(p)},u)$ is defined
		according to (\ref{gDefinition}).
	\end{theorem}
	
	\begin{proof}
		From \textit{Lemma \ref{Gaolemma1}}, since all sub-systems included in the
		Com-$\ell$Sub1 MAS have an equilibrium
		${x^{(p)}}^*\in\mathbb{R}^{n_p}_{\textgreater 0}~(\forall~
		p\in\{0,\cdots,\ell\})$, the Com-$\ell$Sub1 MAS admits a positive
		equilibrium $x^*=\bigotimes_{p=0}^{\ell}
		{x^{(p)}}^{*}\in\mathbb{R}^n_{\textgreater 0}$. Further from
		\textit{Theorem \ref{thm solution s=n}}, $f(x)=\sum_{p=0}^{\ell} f_{p}(x^{(p)})$
		defined by (\ref{solution s=n}) is a twice continuous differentiable solution of the Lyapunov
		Function PDEs (\ref{LyaPDE s=n}) plus (\ref{BoundaryCondition s=n}).
		{From the condition of} $w_p^\top\frac{\partial}{\partial x^{(p)}} g_{p}(x^{(p)},1)\Big|_{x^{(p)}={x^{(p)}}^*}\textless 0$ and {the continuity of $g(x,u)$ with respect to $u$,} there exist neighbourhoods of ${x^{(p)}}^*$, denoted as $\delta({x^{(p)}}^*)$ ($p=1,\cdots,\ell$), such that for all $x^{(p)}\in\delta({x^{(p)}}^*)$ we have
		\begin{equation*}
		w_p^\top\frac{\partial}{\partial x^{(p)}} g_{p}(x^{(p)},\tilde{u}^{(p)}(x^{(p)}))<0,
		\end{equation*}
		where $\tilde{u}^{(p)}(x^{(p)})$ makes $g_{p}(x^{(p)},u)=0$.
		Hence, let $\mu=\bigotimes_{p=0}^{\ell}\mu^{(p)}\in\mathscr{S}$ then for any $x\in\{R^{n_0}_{>0}\bigotimes_{p=1}^{\ell}\delta({x^{(p)}}^*)\}\bigcap(x^*+\mathscr{S})\bigcap R^{n}_{>0}$ we get
		\begin{eqnarray*}
			\mu^{\top} \nabla^{2}f(x) \mu
			&=& {\mu^{(0)}}^{\top}\text{diag}\left\{1/x^{(0)}_{1},\cdots,1/x^{(0)}_{n_{0}}\right\}\mu^{(0)}\\
			&+&\sum_{p=1}^\ell({\mu^{(p)}}^{\top}w_p)^{2} \cdot
			\frac{-w_p^\top\frac{\partial }{\partial x^{(p)}}
				g_p(x^{(p)},\tilde{u}^{(p)})
				/\frac{\partial}{\partial u^{(p)}} g_p(x^{(p)},\tilde{u}^{(p)})
			}
			{\tilde{u}(x^{(p)})} \geq 0.
		\end{eqnarray*}
		Clearly, the above equality holds if and only if $\mu=\mathbbold{0}_n$. This means that all conditions in \textit{Theorem \ref{thm lya function}} are satisfied, and the result is thus shown.
	\end{proof}
	
	\begin{example}
		Consider a Com-$\ell$Sub1 MAS ($\ell=1$) with the reaction route
		following
		\begin{equation}
		\begin{array}{c:c}
		~~\xymatrix{                & S_{2}^{(0)} \ar[dr]^{k_{2}^{(0)}}             \\
			S_{1}^{(0)} \ar[ur]^{k_{1}^{(0)}} & & S_{3}^{(0)}\ar[ll]_{k_{3}^{(3)}} }
		~~&~~
		\xymatrix{ S_{1}^{(1)}  \ar[r]^{k_{1}^{(1)}} &S^{(1)}_{2},\\
			2S^{(1)}_{2} \ar[r]^{k_{2}^{(1)}} &2S^{(1)}_{1}.}~~
		\end{array}
		\end{equation}
		The sub-system
		{$(\mathcal{S}^{(0)},\mathcal{C}^{(0)},\mathcal{R}^{(0)},\mathcal{K}^{(0)})$} is
		complex balanced that has complexes as
		\begin{equation*}
		v_{\cdot1^{(0)}}=v'_{\cdot 3^{(0)}}= \left(
		\begin{array}{c}
		1\\
		0\\
		0
		\end{array}
		\right),~ v_{\cdot2^{(0)}}=v'_{\cdot 1^{(0)}}= \left(
		\begin{array}{c}
		0\\
		1\\
		0
		\end{array}
		\right),~ v_{\cdot3^{(0)}}=v'_{\cdot 2^{(0)}}= \left(
		\begin{array}{c}
		0\\
		0\\
		1
		\end{array}
		\right)
		\end{equation*}
		and admits a {complex balanced equilibrium}
		${x^{(0)}}^*=\left(k_2^{(0)}k_3^{(0)},k_1^{(0)}k_3^{(0)},k_1^{(0)}k_2^{(0)}\right)^{\top}$.
		In addition, the sub-system
		{$(\mathcal{S}^{(1)},\mathcal{C}^{(1)},\mathcal{R}^{(1)},\mathcal{K}^{(1)})$} (the
		same as in Example 2) is of $1$-dimensional stoichiometric subspace
		with complexes
		\begin{equation*}
		v_{\cdot 1^{(1)}}= \left(
		\begin{array}{c}
		1\\ 0
		\end{array}
		\right),~~ v'_{\cdot 1^{(1)}}= \left(
		\begin{array}{c}
		0\\ 1
		\end{array}
		\right),~~ v_{\cdot 2^{(1)}}= \left(
		\begin{array}{c}
		0\\2
		\end{array}
		\right),~~ v'_{\cdot 2^{(1)}}= \left(
		\begin{array}{c}
		2\\ 0
		\end{array}
		\right)
		\end{equation*}
		and an equilibrium
		${x^{(1)}}^{*}=\left(2k^{(1)}_{2},\sqrt{k^{(1)}_{1}}\right)^{\top}$.
		Moreover, the naive boundary complex set is set as
		\begin{equation*}
		\bar{C}_{\bar{x}^{(1)}}= \left\{
		\begin{array}{ll}
		\left\{v_{\cdot 1^{(1)}},v'_{\cdot 2^{(1)}}\right\}, & \bar{x}^{(1)}=(\bar{x}_{{1}^{(1)}},0)^{\top}~\mathrm{with}~\bar{x}_{{1}^{(1)}}>0;\\
		\left\{v_{\cdot 2^{(1)}},v'_{\cdot 1^{(1)}}\right\}, &
		\bar{x}^{(1)}=(0,\bar{x}_{{2}^{(1)}})^{\top}~\mathrm{with}~\bar{x}_{{2}^{(1)}}>0.
		\end{array}
		\right.
		\end{equation*}
		By \textit{Theorem \ref{thm solution s=n}}, there exists a solution supporting the
		Lyapunov Function PDEs induced from this Com-$\ell$Sub1 MAS, written
		as
		$$f(x)=f_0(x^{(0)})+f_1(x^{(1)}),$$
		where $x^{(0)}=(x_{1^{(0)}},x_{2^{(0)}},x_{3^{(0)}})^\top$,
		$x^{(1)}=(x_{1^{(1)}},x_{2^{(1)}})^\top$, $x=x^{(0)}\bigotimes
		x^{(1)}$,
		$$f_0(x^{(0)})={x^{(0)}}^\top\mathrm{Ln}\left(\frac{x^{(0)}}{{x^{(0)}}^*}\right)-\mathbbm{1}_3^\top\left(x^{(0)}-{x^{(0)}}^*\right)$$
		and
		$$f_1(x^{(1)})=\int_{0}^{\gamma\left(x^{(1)}\right)}
		\ln{\frac{k^{(1)}_2\hat{y}^{v_{\cdot
						2^{(1)}}}+\sqrt{\left(k^{(1)}_2\right)^2\hat{y}^{2v_{\cdot
							2^{(1)}}}+4k^{(1)}_1k^{(1)}_2\hat{y}^{v_{\cdot 1^{(1)}}+v_{\cdot
							2^{(1)}}}}}{2k^{(1)}_1\hat{y}^{v_{\cdot 1^{(1)}}}}} \dd \tau.$$ In
		the above equation, $\hat{y}=y^\dag\left(x^{(1)}\right)+\tau w_1$
		and
		\begin{equation*}
		y^\dag\left(x^{(1)}\right)=\frac{1}{2}\left(
		\begin{array}{c}
		\mathbbm{1}_2^\top\\
		\mathbbm{1}_2^\top \\
		\end{array}
		\right)x^{(1)},~~ \gamma\left(x^{(1)}\right)=\frac{w_{1}^\top
			x^{(1)}}{2},~~ w_{1}=(-1,1)^{\top}.
		\end{equation*}
		Further, from
		$g_{1}\left(x^{(1)},u\right)=k^{(1)}_1x_{1^{(1)}}-k^{(1)}_2\left(x_{2^{(1)}}\right)^{2}u^{-1}-k^{(1)}_2\left(x_{2^{(1)}}\right)^2u^{-2}$
		we have
		\begin{equation*}
		w_1^\top\frac{\partial }{\partial
			x^{(1)}}g_{1}\left(x^{(1)},u\right)\Big|_{x^{(1)}={x^{(1)}}^*}=-k_1^{(1)}-4k_2^{(1)}x_{2^{(1)}}^*\textless
		0.
		\end{equation*}
		Therefore, based on \textit{Theorem \ref{thm stable s=n}}, $f(x)$ is {an} available
		Lyapunove function for the given Com-$\ell$Sub1 CRN and could
		suggest its equilibrium
		$\left({x^{(0)}}^{*},{x^{(1)}}^{*}\right)^\top$ to be locally
		asymptotically stable.
	\end{example}
	
	\subsection{Other Two Special CRNs with $\text{dim}\mathscr{S}\geq 2$}
	\begin{example}
	Consider a CRN of the form
	\begin{eqnarray}{\label{CRN s=2}}
	2S_{1}\stackrel{k_1}{\longrightarrow} S_{1}+S_{2},\notag \quad
	2S_{2}\stackrel{k_2}{\longrightarrow} S_{2}+S_{3}, \quad
	2S_{3}\stackrel{k_3}{\longrightarrow} S_{3}+S_{1}.\notag
	\end{eqnarray}
	{We have the species set $\mathcal{S}=\{S_1,S_2,S_3\}$, complex set $\mathcal{C}=\{v_{\cdot 1},v'_{\cdot 1},v_{\cdot 2},v'_{\cdot 2},v_{\cdot 3},v'_{\cdot 3}\}$, reaction set $\mathcal{R}=(v_{\cdot 1}\to v'_{\cdot 1},v_{\cdot 2}\to v'_{\cdot 2},v_{\cdot 3}\to v'_{\cdot 3})$ and kinetic set $\mathcal{K}=(k_1,k_2,k_3)$, where}
	\begin{equation*}
	v_{\cdot 1}=
	\left(
	\begin{array}{c}
	2\\0\\0\\
	\end{array}
	\right),
	v'_{\cdot 1}=
	\left(
	\begin{array}{c}
	1\\1\\0\\
	\end{array}
	\right),
	v_{\cdot 2}=
	\left(
	\begin{array}{c}
	0\\2\\0\\
	\end{array}
	\right),
	v'_{\cdot 2}=
	\left(
	\begin{array}{c}
	0\\1\\1\\
	\end{array}
	\right),
	v_{\cdot 3}=
	\left(
	\begin{array}{c}
	0\\0\\2\\
	\end{array}
	\right),
	v'_{\cdot 3}=
	\left(
	\begin{array}{c}
	1\\0\\1\\
	\end{array}
	\right),
	\end{equation*}
	{By the mass-action kinetics, the dynamics of the system is expressed as}
	\begin{equation}{\label{DynamicsS2}}
	\dot{x}(t)=\sum_{i=1}^{3} k_{i}x^{v_{\cdot i}}\left(v'_{\cdot i} -
	v_{\cdot i}\right),~~x\in\mathbb{R}^n_{\textgreater 0}.
	\end{equation}
	{Therefore,} we have $\text{dim}\mathscr{S}=\text{dim~span}(v'_{\cdot 1}-v_{\cdot
		1},v'_{\cdot 2}-v_{\cdot 2},v'_{\cdot 3}-v_{\cdot 3})=2$.
	Obviously,
	this network does not belong to any type of CRNs mentioned above.
	However, it has some special properties.
	
	\begin{lemma}
		In every positive stoichiometric compatibility class induced by any
		given $\tilde{x}\in\mathbb{R}^n_{\textgreater 0}$ and the
		stoichiometric subspace $\mathscr{S}$ of the MAS governed by
		(\ref{DynamicsS2}), any state $x$ of this MAS is constrained by
		$\mathbbm{1}_3^\top x=\mathbbm{1}_3^\top\tilde{x}$.
	\end{lemma}
	
	\begin{proof}
		From $\mathscr{S}=\text{span}(v'_{\cdot 1}-v_{\cdot 1},v'_{\cdot
			2}-v_{\cdot 2},v'_{\cdot 3}-v_{\cdot 3})$, we have that the
		orthogonal complement space $\mathscr{S}^{\perp}$ of $\mathscr{S}$
		is of one dimension, and $(1,1,1)^{\top}$ can act as a set of bases
		of $\mathscr{S}^{\perp}$, i.e., $(1,1,1)^{\top}\perp\mathscr{S}$.
		Therefore, for every positive stoichiometric compatibility class
		$(\tilde{x}+\mathscr{S})\bigcap\mathbb{R}^{n}_{>0}$ induced by any
		given $\tilde{x}\in\mathbb{R}^n_{\textgreater 0}$ and $\mathscr{S}$,
		the state $x$ of the MAS satisfies $x-\tilde{x}\in\mathscr{S}$,
		i.e., $(x-\tilde{x})\perp(1,1,1)^{\top}$. We get $\mathbbm{1}_3^\top
		x=\mathbbm{1}_3^\top\tilde{x}$.
	\end{proof}
	
	\begin{lemma}
		The MAS governed by (\ref{DynamicsS2}) admits a unique equilibrium
		in each positive stoichiometric compatibility class. Furthermore,
		this sole equilibrium, denoted by $x^*\in\mathbb{R}^3_{\textgreater
			0}$, satisfies
		$\sqrt{k_{1}}x^{*}_{1}=\sqrt{k_{2}}x^{*}_{2}=\sqrt{k_{3}}x^{*}_{3}$.
	\end{lemma}
	
	\begin{proof}
		Let $(\tilde{x}+\mathscr{S})\bigcap\mathbb{R}^{n}_{>0}$ represent
		any positive stoichiometric compatibility class in which the state
		of the MAS following (\ref{DynamicsS2}) evolves. If the equilibrium
		$x^*\in\mathbb{R}^3_{\textgreater 0}$ {exists}, then it must satisfy
		\begin{equation*}
		\mathbbm{1}_3^\top x^*=\mathbbm{1}_3^\top\tilde{x}~~~\text{and}~~~k_{1}\left(x^{*}_{1}\right)^{2}=k_{2}\left(x^{*}_{2}\right)^{2}=k_{3}\left(x^{*}_{3}\right)^{2}.
		\end{equation*}
		Denote
		\begin{equation*}
		K=\left(
		\begin{array}{ccc}
		1 & 1 & 1\\
		\sqrt{k_{1}} & -\sqrt{k_{2}} & 0\\
		\sqrt{k_{1}} & 0 & -\sqrt{k_{3}} \\
		\end{array}
		\right),
		\end{equation*}
		then the above two relations can be integrated together and
		rewritten as
		\begin{equation*}
		Kx^*=\mathbbm{1}_3^\top\tilde{x}.
		\end{equation*}
		Since $\det(K)=\sqrt{k_1k_2}+\sqrt{k_2k_3}+\sqrt{k_1k_3}\neq 0$,
		$x^*$ {exists} and {is also} unique in
		$(\tilde{x}+\mathscr{S})\bigcap\mathbb{R}^{n}_{>0}$. Furthermore,
		$x^*$ supports the relation
		$\sqrt{k_{1}}x^{*}_{1}=\sqrt{k_{2}}x^{*}_{2}=\sqrt{k_{3}}x^{*}_{3}$,
		and
		\begin{equation*}
		x^*=\frac{\mathbbm{1}_3^\top\tilde{x}}{\sqrt{k_{1}k_{2}}+\sqrt{k_{2}k_{3}}+\sqrt{k_{1}k_{3}}}\left(\sqrt{k_{2}k_{3}},
		\sqrt{k_{1}k_{3}}, \sqrt{k_{1}k_{2}}\right)^{\top}.
		\end{equation*}
	\end{proof}
	
	Based on these two properties, it is not difficult to find a
	solution for the Lyapunov function PDEs (\ref{LyaPDEs}) and
	(\ref{BoundaryCondition}) of this MAS.
	
	\begin{theorem}
		For the MAS described by (\ref{DynamicsS2}), if the boundary complex
		set $\mathcal{C}_{\bar{x}}$ is set to be empty, then the Lyapunov
		function PDEs (\ref{LyaPDEs}) and (\ref{BoundaryCondition})
		generated by this MAS admit a solution
		\begin{equation}{\label{LyaSolutionS2}}
		f(x)=2x^\top\mathrm{Ln}\left(\frac{x}{x^*}\right)-2\mathbbm{1}_3^\top
		(x-x^*),
		\end{equation}
		where $x^{*}\in\mathbb{R}^3_{\textgreater 0}$ is an equilibrium of
		the MAS under consideration.
		Moreover this solution can behave as the Lyapunov function to suggest the MAS locally asymptotically stable at $x^{*}$.
	\end{theorem}
	
	\begin{proof}
		Substituting $\nabla f(x)=2\mathrm{Ln}\left(\frac{x}{x^*}\right)$
		into the L.H.S. of (\ref{LyaPDEs}) yields
		\begin{eqnarray*}
			\text{L.H.S~of~Eq.~(\ref{LyaPDEs})}&=&
			k_{1}x^{2}_{1}+k_{2}x_{2}^{2}+k_{3}x_{3}^{2}-k_{1}\left(\frac{x_{1}^{*}}{x_{2}^{*}}\right)^2x_{2}^{2}
			-k_{2}\left(\frac{x_{2}^{*}}{x_{3}^{*}}\right)^2x_{3}^{2}-k_{3}\left(\frac{x_{3}^{*}}{x_{1}^{*}}\right)^2x_{1}^{2}\\
			&=&
			k_{1}x^{2}_{1}+k_{2}x_{2}^{2}+k_{3}x_{3}^{2}-k_{2}x_{2}^{2}-k_{3}x_{3}^{2}-k_{1}x^{2}_{1}
			\\ &=& 0.
		\end{eqnarray*}
		Hence,
		$f(x)=2x^\top\mathrm{Ln}\left(\frac{x}{x^*}\right)-2\mathbbm{1}_3^\top$
		is a solution of the PDE (\ref{LyaPDEs}).
		In the meanwhile, the boundary
		condition of (\ref{BoundaryCondition}) is naturally true with
		$\mathcal{C}_{\bar{x}}=\emptyset$.
		Moreover, the Hessian matrix of $f(x)$ is expressed as
		\begin{equation*}
		\nabla^{2} f(x) =
		\left(
		\begin{array}{ccc}
		\frac{2}{x_{1}} && \\
		&\frac{2}{x_{2}}&\\
		&&\frac{2}{x_{3}}\\
		\end{array}
		\right)
		\end{equation*}
		which is positive definite in $\mathbb{R}^n_{\textgreater 0}$. Hence, the condition \eqref{eq. strictly convex condition} is satisfied for every state, and the asymptotic stability holds immediately from \textit{Theorem \ref{thm lya function}}.
	\end{proof}
	
	\begin{remark}
		The solution
		$f(x)=2x^\top\mathrm{Ln}\left(\frac{x}{x^*}\right)-2\mathbbm{1}_3^\top
		(x-x^*)$ is actually a function {similar to} the {pseudo-Helmholtz free energy function}. It can
		be rewritten as $f(x)=2G(x)$.
	\end{remark}
\end{example}	

\begin{example}
	The second CRN {has $3$-dimensional stoichiometric subspace} and is given by
	\begin{eqnarray}
	3S_{1} \stackrel{k_1}{\longrightarrow} 2S_1+S_{2}, \notag \quad
	2S_{2} \stackrel{k_2}{\longrightarrow} S_{2}+S_3, \notag \quad
	S_3 \stackrel{k_3}{\longrightarrow} S_1, \notag \quad
	0\stackrel{k_4}{\longrightarrow} S_{3}\stackrel{k_5}{\longrightarrow} 0\notag
	\end{eqnarray}
	{with the complexes being}
	\begin{equation*}
	v_{\cdot 1}=
	\left(
	\begin{array}{c}
	3\\0\\0\\
	\end{array}
	\right),
	v'_{\cdot 1}=
	\left(
	\begin{array}{c}
	2\\1\\0\\
	\end{array}
	\right),
	v_{\cdot 2}=
	\left(
	\begin{array}{c}
	0\\2\\0\\
	\end{array}
	\right),
	v'_{\cdot 2}=
	\left(
	\begin{array}{c}
	0\\1\\1\\
	\end{array}
	\right),
	v_{\cdot 3}=
	\left(
	\begin{array}{c}
	0\\0\\1\\
	\end{array}
	\right),
	v'_{\cdot 3}=
	\left(
	\begin{array}{c}
	1\\0\\0\\
	\end{array}
	\right),
	\end{equation*}
	\begin{equation*}
	v_{\cdot 4}= v'_{\cdot 5}=
	\left(
	\begin{array}{c}
	0\\0\\0\\
	\end{array}
	\right),
	v_{\cdot 5}= v'_{\cdot 4}=
	\left(
	\begin{array}{c}
	0\\0\\1\\
	\end{array}
	\right).
	\end{equation*}
{By setting $k_1=k_2=k_3=k_4=k_5=1$ {for simplicity}, the dynamical equation follows}
\begin{equation*}
	\left\{
	\begin{array}{rcl}
	    \dot{x_1}(t) &=& -x_{1}^{3}+x_3,\\
	    \dot{x_2}(t) &=& x_1^3-x_2^{2},  \\
	    \dot{x_3}(t) &=& x_2^2-2x_3+1
	\end{array}
	\right.
\end{equation*}
{and the Lypunov Function PDE is}
\begin{eqnarray}
	&&x_{1}^{3}+x_2^{2}+2x_3+1-x_{1}^{3}\exp\left\{- \frac{\partial f}{\partial x_{1}} +\frac{\partial f}{\partial x_{2}} \right\}
	- x_{2}^{2}\exp\left\{- \frac{\partial f}{\partial x_{2}} +\frac{\partial f}{\partial x_{3}} \right\} -x_{3}\exp\left\{\frac{\partial f}{\partial x_{1}}-\frac{\partial f}{\partial x_{3}} \right\} \notag \\
	&&-\exp\left\{ \frac{\partial f}{\partial x_{3}} \right\} - x_{3}\exp\left\{- \frac{\partial f}{\partial x_{3}} \right\} =0. \notag
\end{eqnarray}
{We choose the boundary complex to be empty set at any boundary point, then the boundary condition \mbox{\eqref{BoundaryCondition}} vanishes.}

{It is not difficult to verify that the following function}  $$f(x)=3\left(x_{1}\ln x_{1}-x_{1}\right)+2\left(x_{2}\ln x_{2}-x_{2}\right)+\left(x_{3}\ln x_{3}-x_{3}\right)$$
{is a solution of the above Lyapunov Function PDE. Moreover, this solution meets all conditions given in \textit{Theorem} \mbox{\ref{thm lya function}}, so it can work for analyzing the asymptotic stability of the MAS. }

\end{example}

 \subsection{Computational verification for a $2$-dimensional CRN}
 {The last CRN has the form}
 \begin{equation*}
 	2S_1 \stackrel{k_1}{\longrightarrow}  S_2 \stackrel{k_2}{\longrightarrow}  S_1 \stackrel{k_3}{\longleftarrow}  0,
 \end{equation*}
{where the complexes are}
 \begin{equation*}
 v_{\cdot 1}=
 \left(
 \begin{array}{c}
 2\\ 0
 \end{array}
 \right),~~
 v'_{\cdot 1}=
 \left(
 \begin{array}{c}
 0\\ 1
 \end{array}
 \right),~~
 v_{\cdot 2}=
 \left(
 \begin{array}{c}
 0\\1
 \end{array}
 \right),~~ v'_{\cdot 2}= \left(
 \begin{array}{c}
 1\\ 0
 \end{array}
 \right),
 ~~
 v_{\cdot 3}=
 \left(
 \begin{array}{c}
 0\\0
 \end{array}
 \right),~~ v'_{\cdot 3}= \left(
 \begin{array}{c}
 1\\ 0
 \end{array}
 \right).
 \end{equation*}
 {We also set $k_1=k_2=k_3=1$ {for simplicity}, and thus write the dynamics as}
  \begin{equation*}
 	\left\{
 	\begin{array}{cl}
 	\dot{x_1}(t) & =-2x_1^2+x_2+1, \\ \dot{x_{2}}(t) & =x_1^2-x_2.
 	\end{array}
 	\right.
 \end{equation*}
{Although this CRN looks simple and the dimension of its stoichiometric subspace is only $2$, to our knowledge it is difficult to find a solution for its Lyapunov Function PDE}
 \begin{equation}\label{Lya pde in the numerical example}
 	x_1^2+x_2+1
 	-x_1^2\exp\left\{-2\frac{\partial f}{\partial x_1}+\frac{\partial f}{\partial x_2}\right\}
 	-x_2\exp\left\{\frac{\partial f}{\partial x_1}-\frac{\partial f}{\partial x_2}\right\}
 	-\exp\left\{\frac{\partial f}{\partial x_1}\right\}
 	=0.
 \end{equation}
 {We thus try to make a computational verification for this example. }

 {Note that $(x_1^*,x_2^*)^\top=(1,1)^\top$ is the unique equilibrium in the network system, and there are three types of native boundary complex sets according to the boundary points considered, that is}
 \begin{equation*}
\bar{C}_{\bar{x}}= \left\{
\begin{array}{ll}
\left\{(0,0)^{\top},(0,1)^{\top}\right\}, & \bar{x}=(\bar{x}_{1},0)^{\top}~\mathrm{with}~\bar{x}_{1}>0;\\
\left\{(0,0)^{\top},(1,0)^{\top},(2,0)^{\top}\right\}, &
\bar{x}=(0,\bar{x}_{2})^{\top}~\mathrm{with}~\bar{x}_{2}>0; \\
\{(0,0)^{\top}\}, & \bar{x}=(0,0)^{\top}.
\end{array}
\right.
\end{equation*}
{In order to observe whether $f(x)$ {has the potential} to act as the Lyapunov function, we make a Taylor expansion about it at the equilibrium $(1,1)^\top$. For simplicity but without loss of generality, the expansion is made up to the third order. Fig. \mbox{\ref{Fig. numerical example}} exhibits the simulation results about $f(x)$ in sub-figure (a) and about the minimum eigenvalue, denoted by $\lambda_\text{min}$, of its Hessian matrix in sub-figure (c). From the sub-figure (a) and the corresponding contours sub-figure (b), it is suggested that $f(x)$ is convex with the minimum evaluated at the equilibrium. Further from the sub-figure (c) and the corresponding contours sub-figure (d), there exists a neighbourhood around the equilibrium in which $f(x)$ is strictly convex. This indicates that $f(x)$ meets all conditions requested in \textit{Theorem} \mbox{\ref{thm lya function}}. Hence, the computational simulation also supports that the Lyapunov Function PDE method is valid. }

\begin{figure}
 	\centering
 	\includegraphics[width=0.9\linewidth]{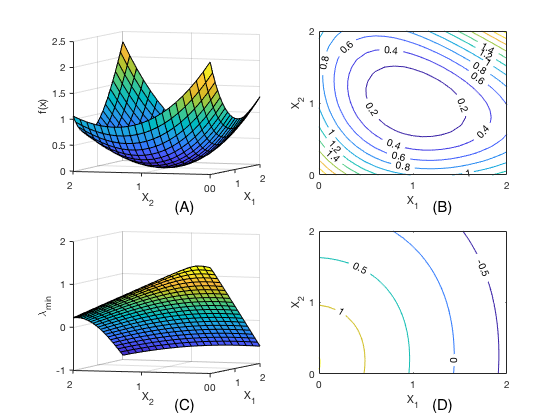}
 	\caption{{Simulation results for the Lyapunov Function PDE \mbox{(\ref{Lya pde in the numerical example})}: (a) $f(x)$; (b) the contours of $f(x)$; (c) $\lambda_{\mathrm{min}}$; (d) the contours of $\lambda_{min}$. }}
 	{\label{Fig. numerical example}}
 \end{figure}
	
	\section{Conclusion and a Conjecture}
	This paper is devoted to developing the Lyapunov function with clear
	physical meaning for stability analysis for chemical reaction
	networks. We have attempted to address this issue by establishing
	approximation from a microscopic concept of CRNs, the scaling
	non-equilibrium potential, to the macroscopic notation of the
	candidate Lyapunov function. After rewriting the Chemical Master
	Equation skillfully, we have succeeded in implementing the
	approximation and transformed the ODE into a PDE, which together
	with the developed boundary condition yields the Lyapunov Function
	PDEs. And then, we have proved that the solution (if exists) of the
	PDEs is dissipative, and thus has great potential to become a
	Lyapunov function. {Next, we have applied the Lyapunov Function PDEs
	to complex-balanced CRNs and general CRNs with $1$-dimensional
	stoichiometric subspace. For both cases, we construct their solutions that can act as Lyapunov functions rendering the respective system to be locally asymptotically stable.}
	Finally, we
	have extended the applications of the Lyapunov Function PDEs to some
	special CRNs with more than $2$-dimensional stoichiometric subspace,
	and {showed} that the PDEs also work validly for them in stability
	analysis. %by generating a solution as an available Lyapunov function.
	
	Notwithstanding the performance illustrated by the Lyapunov Function
	PDEs is very encouraging, there are still some problems needed to be
	explored in the future. One of the most urgent problems is to prove
	that the Lyapunov Function PDEs CAN or CANNOT serve for general CRNs
	with more than $2$-dimensional stoichiometric subspace. This may be
	an extremely arduous task, however, we are inclined to think they
	can. We summarize the proof task as a conjecture: ``\textit{For any
		mass action system that admits a stable positive equilibrium,
		if the boundary complex set is equipped properly, then the Lyapunov
		Function PDEs induced by this system have a solution qualified as a
		Lyapunov function to suggest the system to be locally
		asymptotically stable at the equilibrium."}	{The converse problem is also interesting, i.e., will all Lyapunov functions be solutions to the PDEs in some sense?}

%\bibliographystyle{siamplain}
%\bibliography{our_paper}

\end{document}